\documentclass{amsart}
\usepackage[T1]{fontenc}   

\usepackage[colorlinks=true, urlcolor=blue, citecolor=blue, linkcolor=blue, hyperfootnotes=true]{hyperref}
\usepackage{float}
\usepackage{amssymb}
\usepackage{mathrsfs}
\usepackage{dsfont}
\usepackage{enumitem}
\usepackage{aliascnt}

\allowdisplaybreaks

\newcommand\barrow{\textstyle\mathop{\rightarrow}_{}^{\hspace{-8pt}\bullet}}
\newcommand\arrowb{\textstyle\mathop{\rightarrow}_{\hspace{-8pt}\bullet}^{}}
\newcommand\carrow{\textstyle\mathop{\rightarrow}_{}^{\hspace{-8pt}\circ}}
\newcommand\arrowc{\textstyle\mathop{\rightarrow}_{\hspace{-8pt}\circ}^{}}
\newcommand\barrowc{\textstyle\mathop{\rightarrow}_{\hspace{-8pt}\circ}^{\hspace{-8pt}\bullet}}
\newcommand\carrowc{\textstyle\mathop{\rightarrow}_{\hspace{-8pt}\circ}^{\hspace{-8pt}\circ}}
\newcommand\carrowb{\textstyle\mathop{\rightarrow}_{\hspace{-8pt}\bullet}^{\hspace{-8pt}\circ}}

\newtheorem{thm}{Theorem}[section]

\newaliascnt{prp}{thm}
\newtheorem{prp}[prp]{Proposition}
\aliascntresetthe{prp}

\newaliascnt{cor}{thm}
\newtheorem{cor}[cor]{Corollary}
\aliascntresetthe{cor}

\newaliascnt{lem}{thm}

\aliascntresetthe{lem}

\theoremstyle{definition}

\newaliascnt{dfn}{thm}
\newtheorem{dfn}[dfn]{Definition}
\aliascntresetthe{dfn}

\newaliascnt{xpl}{thm}
\newtheorem{xpl}[xpl]{Example}
\aliascntresetthe{xpl}

\numberwithin{equation}{section}

\author{Tristan Bice}
\address{Institute of Mathematics of the Polish Academy of Sciences\\
Warsaw\\
Poland}
\email{Tristan.Bice@gmail.com}
\thanks{This research has been supported by IMPAN (Poland).}

\title{Distance Domains: Completeness}

\begin{document}

\renewcommand{\bf}{\bfseries}
\renewcommand{\sc}{\scshape}
\vspace{0.5in}

\begin{abstract}
We explore extensions of domain theoretic concepts, replacing transitive relations with general non-symmetric distances.  These lead to a generalization of Smyth completeness which we characterize in various ways analogous to our previous Yoneda completeness characterizations.
\end{abstract}

\maketitle

\section*{\bf Motivation}

A number of works have extended domain theory \textendash\, see \cite{GierzHofmannKeimelLawsonMisloveScott2003} \textendash\, from posets to more metric-like structures.  However, both the classical theory and these generalizations tend to focus on just one aspect of the dual nature of domains.  Our primary goal is explore the other aspect.

More precisely, the standard approach to domain theory is to start with a partial order $\leq$ and then define its way-below relation $\ll$, a transitive but generally non-reflexive relation.  An alternative approach is to start with a transitive relation $\ll$ and then define its lower order $\leq$.  Using maxima rather than suprema, one also obtains dual notions of completeness and continuity for $\ll$.  This is the approach we generalize, working with a general non-symmetric distance $\mathbf{d}$ and its lower hemimetric $\underline{\mathbf{d}}$.

Also, previous works have developed quantitative domain theory in a highly category or fuzzy theoretic way \textendash\, see e.g. \cite{HofmannWaszkiewicz2011} and \cite{LiRao2013}.  Another goal of our paper is to provide a more classic approach through topology, metric and order theory, building on \cite{Goubault2013}.  This leads to certain natural generalizations and should also be more accessible to analysts.

In particular, we have two examples in mind from non-commutative topology.  First, consider the hereditary C*-subalgebras $\mathcal{H}(A)$ of a C*-algebra $A$, ordered by inclusion $\subseteq$.  When $A$ is commutative, these correspond to the open subsets of a locally compact Hausdorff topological space, a well-known example of a classical domain.  However, $\mathcal{H}(A)$ may fail to be a domain in general, even for basic non-commutative C*-algebras like $C([0,1],M_2)(=$continuous functions from the unit interval to two by two complex matrices$)$.  The key observation here is that $\mathcal{H}(A)$ does, however, always form a distance domain when we replace the inclusion ordering $\subseteq$ with the Hausdorff distance $\underline{\mathbf{d}}$ on the positive unit balls $B^1_+$,
\[\underline{\mathbf{d}}(B,C)=\sup_{b\in B^1_+}\inf_{c\in C^1_+}\|b-bc\|.\]
Here the way-below distance $\mathbf{d}$ comes from the reverse Hausdorff distance
\[\mathbf{d}(B,C)=\inf_{c\in C^1_+}\sup_{b\in B^1_+}\|b-bc\|.\]
Incidentally, $(b,c)\mapsto\|b-bc\|$ is itself a natural example of a non-hemimetric distance on $A^1_+$ \textendash\, see \cite[Proposition 2.3]{BiceVignati2018}.

There can also be merit in quantifying classical domains, e.g. consider the lower semicontinuous [0,1]-valued functions $LSC(X,[0,1])$ on some compact Hausdorff $X$ with the pointwise ordering $\leq$.  This is another well-known example of a classical domain \textendash\, see \cite[Example I-1.22]{GierzHofmannKeimelLawsonMisloveScott2003}.  But when we replace $\leq$ with 
\[\underline{\mathbf{d}}(f,g)=\sup_{x\in X}(f(x)-g(x))_+,\]
we get an even nicer structure.  Specifically $LSC(X,[0,1])$ becomes an \emph{algebraic} domain, in an appropriate quantitative sense, where the finite/compact elements \textendash\, see \cite[Definition 7.4.56]{Goubault2013} \textendash\, are precisely the continuous functions $C(X,[0,1])$ (by a slight generalization of Dini's theorem).  Moreover, this extends to the lower semicontinuous elements of $A^{**1}_+$ for a much larger class of ordered Banach spaces $A$ \textendash\, see \cite{Bice2016c}.

Apart from the inherent interest in generalization, we feel examples like this justify the study of distance domains.  So from now on we put functional analysis to one side to develop a general domain theory for non-symmetric distances.

\section*{\bf Outline}

While category theory is not our focus, we do consider one very elementary category $\mathbf{GRel}$ of generalized relations.  Indeed, throughout we make use of various interpolation assumptions which are concisely described by composition $\circ$ in $\mathbf{GRel}$. In \autoref{TheCat}, we describe the basic properties of $\mathbf{GRel}$ and set out much of the notation used throughout.  Note our functions take values in $[0,\infty]$, rather than the more general quantales often considered elsewhere.  This is primarily to reduce the notational burden, which is already quite heavy due to the various topologies, relations and operations we need to consider.  In any case, $[0,\infty]$ valued functions are perfectly suited to the analytic examples we have in mind.

As mentioned above, one of our primary goals is to generalize previous work on hemimetrics to distances, functions merely satisfying the triangle inequality.  This generalization is crucial because we want to develop a dual theory of distance domains starting from distance analogs of the way-below relation.  In \autoref{Distances} we discuss these distances $\mathbf{d}$ and their associated upper and lower hemimetrics $\overline{\mathbf{d}}$ and $\underline{\mathbf{d}}$.

Next, in \autoref{Uniform} we breifly introduce the uniform preorder $\precapprox$ and equivalence relation $\approx$ on generalized relations.  This generalizes the usual uniform equivalence of metrics and is needed to describe weak interpolation assumptions required for the best results (e.g. see \autoref{hemi<d}).

In \autoref{B}, we introduce balls and their associated topologies.  In particular, we show how balls characterize upper and lower hemimetrics and how the preorders $\leq^{\overline{\mathbf{d}}}$ and $\leq^{\underline{\mathbf{d}}}$ defined from $\mathbf{d}$ coincide with the specialization preorders of ball topologies.

As we deal with non-hemimetric distances, it is natural to consider a certain strict version $<^\mathbf{d}$ of $\leq^\mathbf{d}$, which we discuss in \autoref{TSO}.  This will be particularly important in our future work when we exhibit equivalences between distance domains and classical domains of formal balls.  As a preliminary to this, here we investigate the relationship between $\underline{<^\mathbf{d}}$ and $\leq^{\underline{\mathbf{d}}}$ under certain interpolation assumptions.

In \autoref{N}, we make some elementary observations on nets and their limits.  This leads to \autoref{CN}, where we discuss two natural generalizations of Cauchy nets.  Note here, as elsewhere, basic properties of hemimetrics can often be extended to distances by replacing $\mathbf{d}$ with $\overline{\mathbf{d}}$ and $\underline{\mathbf{d}}$ where appropriate.

We also aim to develop the theory in a more topological way.  The key here is to consider topologies generated by open holes as well as balls.  In \autoref{Topology} we characterize convergence in combinations of ball and hole topologies.

Yet another one of our goals is to explore the connection between topological and relational extensions of metric and order theoretic concepts.
\begin{figure}[H]
\caption{Metric vs Order Analogs}\label{Dualities}
\vspace{5pt}\begin{tabular}{|c|c|}
\hline Topological & Relational\\ \hline
Nets & Subsets\\
$\mathbf{d}$-Cauchy & $\mathbf{d}$-directed\\
$\mathbf{d}^\circ_\circ$-limit & $\mathbf{d}$-supremum\\
$\mathbf{d}^\bullet_\circ$-limit & $\mathbf{d}$-maximum\\ \hline
\end{tabular}
\end{figure}
\noindent
As with hole topologies, we feel the relational notions have not received the attention they deserve.  Even apart from their intrinsic interest, these relational notions can serve as a useful intermediary between classical order theoretic concepts and their topological generalizations.  So in \autoref{DirectedSubsets} we define $\mathbf{d}$-directed subsets and explore their relation to $\mathbf{d}$-Cauchy nets.

Suprema are usually considered the poset analog of limits.  However maxima, in an appropriate sense, can be better suited to non-reflexive transitive relations.  In \autoref{UpperBounds} we extend these concepts to distances $\mathbf{d}$ and examine their connection to suprema and maxima relative to $\leq^\mathbf{d}$ and $<^\mathbf{d}$.

In \autoref{Completeness}, we define topological and relational notions of completeness and explain how they generalize standard notions of Yoneda, Smyth, metric and directed completeness.  We then show how to turn $\mathbf{d}$-Cauchy nets into $\mathbf{d}$-directed subsets under several interpolation conditions.  These allow $\mathbf{d}^\bullet_\circ$-completeness(=Smyth completeness for hemimetric $\mathbf{d}$) to be derived from $\mathbf{d}$-$\max$-completeness in \autoref{Sc}, complementing the Yoneda completeness characterizations in \cite{Bice2017}.

In our future work we will discuss generalizations of continuity and the resulting generalizations of domains, in particular showing how to complete (generalized) predomains to domains via the (reverse) Hausdorff distance and the formal ball construction.

\section[blah]{Generalized Relations}\label{TheCat}

The traditional category theoretic approach to quasimetric spaces is to take each quasimetric as its own category, with the elements of the space as objects and the values of the quasimetric as morphisms, as in \cite{Lawvere2002}.  Alternatively, quasimetric spaces are sometimes considered as the objects of a category with Lipschitz maps as morphisms, as in \cite[Defintion 6.2.13]{Goubault2013}.  However, the constructions we consider are best described in a category with quasimetrics, and even more general binary functions, as the morphisms instead.  This is like the category of modules considered in \cite[\S 2.3]{HofmannWaszkiewicz2012}, except that our objects are just sets, without any distinguished hemimetric structure.

Specifically, we consider any $\mathbf{d}\in[0,\infty]^{X\times Y}(=$ functions from $X\times Y$ to $[0,\infty]$) as a \emph{generalized relation} from $X$ to $Y$.  We extend the standard infix notation for classical relations to generalized relations and define
\[x\mathbf{d}y=\mathbf{d}(x,y).\]
Just like the category $\mathbf{Rel}$ of classical relations, generalized relations form the morphisms of a category $\mathbf{GRel}$ when composition $\mathbf{d}\circ\mathbf{e}\in[0,\infty]^{X\times Y}$ of $\mathbf{d}\in[0,\infty]^{X\times Z}$ and $\mathbf{e}\in[0,\infty]^{Z\times Y}$ is defined by
\[x(\mathbf{d}\circ\mathbf{e})y=\inf_{z\in Z}(x\mathbf{d}z+z\mathbf{e}y).\]
In fact, $\mathbf{Rel}$ becomes a wide subcategory of $\mathbf{GRel}$ when we identify each relation $\sqsubset\ \subseteq X\times Y$ with its characteristic function (as we do from now \nolinebreak on):
\[\sqsubset(x,y)=\begin{cases}0&\text{if }x\sqsubset y\\ \infty&\text{otherwise}.\end{cases}\]

For any $\mathbf{d}\in[0,\infty]^{X\times Y}$, $\sqsubset\ \subseteq[0,\infty]^{[0,\infty]}$ and $r\in[0,\infty]$ we define
\[x\sqsubset^\mathbf{d}_ry\qquad\Leftrightarrow\qquad x\mathbf{d}y\sqsubset r.\]
In particular, we let $\leq^\mathbf{d}\ =\ \leq^\mathbf{d}_0$ so
\[x\leq^\mathbf{d}y\qquad\Leftrightarrow\qquad x\mathbf{d}y=0.\]
Equivalently, $\leq^\mathbf{d}$ is the relation identified with $\infty\mathbf{d}$, where $\infty0=0$ and $\infty r=\infty$, for $r>0$.  Note $\mathbf{d}\mapsto\ \leq^\mathbf{d}$ is a left inverse of the inclusion from $\mathbf{Rel}$ to $\mathbf{GRel}$, which is also functorial in that
\[\leq^\mathbf{d}\circ\leq^\mathbf{e}\quad\subseteq\quad\leq^{\mathbf{d}\circ\mathbf{e}}.\]

Various properties of $\mathbf{Rel}$ also extend to $\mathbf{GRel}$.  For example, as in \cite{Tsalenko2001}, $\mathbf{GRel}$ is a category with involution $\mathbf{d}^\mathrm{op}$ defined by
\[x\mathbf{d}^\mathrm{op}y=y\mathbf{d}x.\]
Also, $\mathbf{GRel}$ is a 2-category, namely a $2$-poset, with the pointwise order
\[\mathbf{d}\leq\mathbf{e}\quad\Leftrightarrow\quad\forall x\in X\,\forall y\in Y\ x\mathbf{d}y\leq x\mathbf{e}y,\]
which is compatible with both $\circ$ and $^\mathrm{op}$.  Each hom-set $[0,\infty]^{X\times Y}$ is also a complete lattice with minimum $\mathbf{0}$ and maximum $\boldsymbol{\infty}$ where, for $x\in X$, $y\in Y$ and $r\in[0,\infty]$,
\[x\mathbf{r}y=r.\]
In particular, we have `intersections' $\mathbf{d}\vee\mathbf{e}$ and symmetrizations
\[\mathbf{d}^\vee=\mathbf{d}\vee\mathbf{d}^\mathrm{op},\]
when $X=Y$, in which case we define $\equiv^\mathbf{d}\ =\ (\leq^\mathbf{d})^\vee\ =\ \leq^{\mathbf{d}^\vee}$, i.e.
\[x\equiv^\mathbf{d}y\qquad\Leftrightarrow\qquad x\mathbf{d}y=0=y\mathbf{d}x.\]
In fact, the only thing stopping $\mathbf{GRel}$ from being an allegory, in the sense of \cite{FreydScedrov1990}, is the modularity requirement.

However, as in division allegories, we do have Kan extensions/lifts.  Namely, for $\mathbf{d}\in[0,\infty]^{X\times Z}$ and $\mathbf{e}\in[0,\infty]^{Y\times Z}$, define $\mathbf{d}/\mathbf{e}\in[0,\infty]^{X\times Y}$ \nolinebreak by
\begin{align*}
x(\mathbf{d}/\mathbf{e})y&=\sup_{z\in Z}(x\mathbf{d}z-y\mathbf{e}z)_+,
\end{align*}
where $r_+=r\vee0$, for $r\in[0,\infty]$, and we take $\infty-\infty=0$.  This guarantees
\begin{equation}\label{a<b+c}
a\leq b+c\qquad\Leftrightarrow\qquad a-b\leq c,
\end{equation}
for all $a,b,c\in[0,\infty]$.  It also means that, for all $c\in[0,\infty)$,
\begin{equation}\label{a+(-b+c)}
a+(-b+c)\leq(a-b)+c.
\end{equation}
Also, for $\mathbf{d}\in[0,\infty]^{Z\times Y}$ and $\mathbf{e}\in[0,\infty]^{Z\times X}$, define $\mathbf{e}\backslash\mathbf{d}\in[0,\infty]^{X\times Y}$ by
\begin{align*}
x(\mathbf{e}\backslash\mathbf{d})y&=\sup_{z\in Z}(z\mathbf{d}y-z\mathbf{e}x)_+.
\end{align*}

\begin{prp}\label{Kan}
For $\mathbf{d}\in[0,\infty]^{X\times Z}$, $\mathbf{e}\in[0,\infty]^{Z\times Y}$ and $\mathbf{f}\in[0,\infty]^{X\times Y}$,
\[\mathbf{f}/\mathbf{e}\leq\mathbf{d}\qquad\Leftrightarrow\qquad\mathbf{f}\leq\mathbf{d}\circ\mathbf{e}\qquad\Leftrightarrow\qquad\mathbf{d}\backslash\mathbf{f}\leq\mathbf{e}.\]
\end{prp}

\begin{proof}
Simply note that, for all $x\in X$, $y\in Y$ and $z\in Z$,
\[x\mathbf{f}y-z\mathbf{e}y\leq x\mathbf{d}z\quad\Leftrightarrow\quad x\mathbf{f}y\leq x\mathbf{d}z+z\mathbf{e}y\quad\Leftrightarrow\quad x\mathbf{f}y-x\mathbf{d}z\leq z\mathbf{e}y.\qedhere\]
\end{proof}

\section{\bf Distances}\label{Distances}

We call $\mathbf{d}\in[0,\infty]^{X\times X}$ a \emph{distance}\footnote{\label{dfoot}Functions merely satisfying the triangle inequality do not appear to have been named before.  We feel `distance' is appropriate, as this is already used informally to refer to various functions which at least satisfy the triangle inequality.  But if we were to follow the tradition of adding prefixes to `metric' for weaker notions, `demimetric' or something similar might be appropriate.} if it satisfies the triangle inequality
\[\mathbf{d}\leq\mathbf{d}\circ\mathbf{d}.\tag{$\triangle$}\label{tri}\]
Equivalently, \eqref{tri} is saying that, for all $r,s\in(0,\infty)$ and $x,y,z\in X$,
\[x<^\mathbf{d}_rz<^\mathbf{d}_sy\qquad\Rightarrow\qquad x<^\mathbf{d}_{r+s}y.\]
In particular, $\sqsubset\ \subseteq X\times X$ is a distance iff it is transitive in the usual sense.  As $\mathbf{d}\mapsto\ \leq^\mathbf{d}$ is functorial, this means $\leq^\mathbf{d}$ is transitive whenever $\mathbf{d}$ is a distance.  As in \cite[Definition 6.1.1]{Goubault2013}, we call a distance $\mathbf{d}$ a
\begin{enumerate}
\item \emph{hemimetric} if $\leq^\mathbf{d}$ is a preorder.
\item \emph{quasimetric} if $\leq^\mathbf{d}$ is a partial order.
\end{enumerate}
(Recall that a preorder is a reflexive ($=\ \subseteq\ \leq$) transitive relation and a partial order is an antisymmetric ($\leq\cap\leq^\mathrm{op}\ \subseteq\ =$) preorder).

Non-hemimetric distances have rarely been considered until now.  However, the extra generality is vital if we want to consider distance analogs of non-reflexive transitive relations, like the way-below relation from domain theory.  But there are two closely related hemimetrics associated to any generalized relation, which will be crucial to our later work.

To avoid repetition, we now make the following standing assumption.
\[\textbf{We are given sets $X$ and $Y$ and functions }\mathbf{d},\mathbf{e}\in[0,\infty]^{X\times Y}.\]

\begin{dfn}
\begin{align}
\label{doverdef}\overline{\mathbf{d}}&=\mathbf{d}/\mathbf{d}\in[0,\infty]^{X\times X}&&\text{i.e.}&x\overline{\mathbf{d}}z&=\sup_{y\in Y}(x\mathbf{d}y-z\mathbf{d}y)_+.\\
\label{dunderdef}\underline{\mathbf{d}}&=\mathbf{d}\backslash\mathbf{d}\in[0,\infty]^{Y\times Y}&&\text{i.e.}&z\underline{\mathbf{d}}y&=\sup_{x\in X}(x\mathbf{d}y-x\mathbf{d}z)_+.
\end{align}
We call $\overline{\mathbf{d}}$ and $\underline{\mathbf{d}}$ the \emph{upper} and \emph{lower hemimetric} of $\mathbf{d}$ respectively.
\end{dfn}

This terminology is justified by the following.

\begin{prp}\label{hemiprop}
Both $\overline{\mathbf{d}}$ and $\underline{\mathbf{d}}$ are hemimetrics and $\mathbf{d}=\overline{\mathbf{d}}\circ\mathbf{d}=\mathbf{d}\circ\underline{\mathbf{d}}$.
\end{prp}

\begin{proof} $\mathbf{d}\leq(\mathbin{=}\circ\mathbf{d})$ implies $\overline{\mathbf{d}}=\mathbf{d}/\mathbf{d}\leq\ =$ so $\leq^{\overline{\mathbf{d}}}$ is reflexive.  As $\mathbf{d}/\mathbf{d}\leq\overline{\mathbf{d}}$,
\begin{align*}
\mathbf{d}&\leq\overline{\mathbf{d}}\circ\mathbf{d}\leq(\mathbin{=}\circ\mathbf{d})=\mathbf{d}\quad\text{and}\\
\mathbf{d}&\leq\mathbf{\overline{d}}\circ\mathbf{d}\leq\mathbf{\overline{d}}\circ\mathbf{\overline{d}}\circ\mathbf{d}\quad\text{so}\\
\mathbf{\overline{d}}&=\mathbf{d}/\mathbf{d}\leq\mathbf{\overline{d}}\circ\mathbf{\overline{d}},\quad\text{i.e. $\overline{\mathbf{d}}$ is a distance}.
\end{align*}
Thus $\overline{\mathbf{d}}$ is a hemimetric with $\mathbf{d}=\overline{\mathbf{d}}\circ\mathbf{d}$.  As $\overline{\mathbf{d}^\mathrm{op}}=\underline{\mathbf{d}}^\mathrm{op}$, $\underline{\mathbf{d}}^\mathrm{op}$ and hence $\underline{\mathbf{d}}$ is a hemimetric with $\mathbf{d}^\mathrm{op}=\overline{\mathbf{d}^\mathrm{op}}\circ\mathbf{d}^\mathrm{op}=\underline{\mathbf{d}}^\mathrm{op}\circ\mathbf{d}^\mathrm{op}$ and hence $\mathbf{d}=\mathbf{d}\circ\underline{\mathbf{d}}$.
\end{proof}

\begin{prp}
If $X=Y$ (i.e. $\mathbf{d}\in[0,\infty]^{X\times X}$) then\footnote{The $\Leftarrow$ in \eqref{hemi} is a form of the Yoneda lemma \textendash\, see \cite[Exercise 7.5.26]{Goubault2013}.}
\vspace{-2pt}
\begin{align}
\label{dis}\overline{\mathbf{d}}\leq\mathbf{d}\qquad\Leftrightarrow\qquad\underline{\mathbf{d}}\leq\mathbf{d}\qquad\Leftrightarrow&\qquad\mathbf{d}\text{ is a distance}.\\
\label{ref}\overline{\mathbf{d}}\geq\mathbf{d}\qquad\Leftrightarrow\qquad\underline{\mathbf{d}}\geq\mathbf{d}\qquad\Leftrightarrow&\qquad\leq^\mathbf{d}\text{ is reflexive}.\\
\label{hemi}\overline{\mathbf{d}}=\mathbf{d}\qquad\Leftrightarrow\qquad\underline{\mathbf{d}}=\mathbf{d}\qquad\Leftrightarrow&\qquad\mathbf{d}\text{ is a hemimetric}.
\end{align}
\end{prp}

\begin{proof} We consider $\mathbf{\overline{d}}$, and the $\underline{\mathbf{d}}$ statements then follow from $\overline{\mathbf{d}^\mathrm{op}}=\underline{\mathbf{d}}^\mathrm{op}$.
\begin{itemize}
\item[\eqref{dis}]  $\mathbf{d}\leq\mathbf{d}\circ\mathbf{d}\ \Leftrightarrow\ \mathbf{d}/\mathbf{d}\leq\mathbf{d}$.

\item[\eqref{ref}]  If $\mathbf{d}\leq\mathbf{\overline{d}}$ then $\mathbf{d}\leq\ =$.  If $\mathbf{d}\leq\ =$ then $\mathbf{d}=(\mathbf{d}/\mathbin{=})\leq\mathbf{d}/\mathbf{d}=\mathbf{\overline{d}}$.

\item[\eqref{hemi}]  Immediate from \eqref{dis} and \eqref{ref}.
\qedhere
\end{itemize}
\end{proof}

\begin{xpl}\label{dqxpl}
Consider $\mathbf{f},\mathbf{q}\in[0,1]^{[0,1]\times[0,1]}$ given by
\begin{align*}
x\mathbf{f}y&=x(1-y).\\
x\mathbf{q}y&=(x-y)_+.
\end{align*}
Here $\mathbf{q}$ is the restriction of the usual quasimetric on $[0,\infty]$ (note \eqref{tri} for $\mathbf{q}$ follows from the subadditivity of $_+$) and $\mathbf{f}$ is also a distance as
\[x(1-y)=x(1-z+z)(1-y)\leq x(1-z)+z(1-y).\]
As $(x-y)_+=\sup\limits_{z\in[0,1]}(x(1-z)-y(1-z))_+=\sup\limits_{z\in[0,1]}(z(1-y)-z(1-x))_+$,
\[\mathbf{q}=\overline{\mathbf{f}}=\underline{\mathbf{f}}.\]
\end{xpl}

Before moving on, we make an observation about restrictions.  First, identify $Z\subseteq Y$ with the characteristic function on $Y\times Y$ of $=$ on $Z$, i.e.
\[Z(x,y)=\begin{cases}0&\text{if }x=y\in Z\\ \infty&\text{otherwise},\end{cases}\]
so $\mathbf{d}\circ Z\circ\overline{\mathbf{d}}$ then denotes composition restricted to $Z$, i.e.
\[x(\mathbf{d}\circ Z\circ\overline{\mathbf{d}})y=\inf_{z\in Z}(x\mathbf{d}z+z\overline{\mathbf{d}}y).\]

\begin{prp}\label{reflexrestrict}
If $\mathbf{d}\circ Z\circ\underline{\mathbf{d}}\leq\mathbf{d}$ then $\overline{\mathbf{d}}=\overline{\mathbf{d}|_{X\times Z}}$.
\end{prp}

\begin{proof}
For any $w,x\in X$, we see that
\[w\overline{\mathbf{d}|_{X\times Z}}x=\sup_{z\in Z}(w\mathbf{d}z-x\mathbf{d}z)\leq\sup_{y\in Y}(w\mathbf{d}y-x\mathbf{d}y)=w\overline{\mathbf{d}}x,\]
so $\overline{\mathbf{d}|_{X\times Z}}\leq\overline{\mathbf{d}}$.  Conversely, for any $w,x\in X$,
\begin{align*}
w\overline{\mathbf{d}}x&=\sup_{y\in Y}(w\mathbf{d}y-x\mathbf{d}y)_+\\
&\leq\sup_{y\in Y}(w\mathbf{d}y-x(\mathbf{d}\circ Z\circ\underline{\mathbf{d}})y)_+\\
&=\sup_{y\in Y}(w\mathbf{d}y-\inf_{z\in Z}(x\mathbf{d}z+z\underline{\mathbf{d}}y))_+\\
&=\sup_{y\in Y,z\in Z}(w\mathbf{d}y-x\mathbf{d}z-z\underline{\mathbf{d}}y)_+\\
&\leq\sup_{z\in Z}(w\mathbf{d}z-x\mathbf{d}z)_+\\
&=w\overline{\mathbf{d}|_{X\times Z}}x,
\end{align*}
where $w\mathbf{d}y\leq w\mathbf{d}z+z\underline{\mathbf{d}}y$ follows from $\mathbf{d}=\mathbf{d}\circ\underline{\mathbf{d}}$, by \autoref{hemiprop}.
\end{proof}

\section{The Uniform Preorder}\label{Uniform}

As mentioned above, we usually view $\mathbf{GRel}$ as a $2$-poset with respect to the pointwise ordering on morphisms.  However, there is also a weaker $2$-proset structure based on the notion of uniform equivalence for metrics.  Specifically, we define the \emph{uniform preorder} $\precapprox$ by
\[\mathbf{d}\precapprox\mathbf{e}\qquad\Leftrightarrow\qquad\forall Z\subseteq X\times Y\ (\inf_{(x,y)\in Z}x\mathbf{e}y=0\quad\Rightarrow\quad\inf_{(x,y)\in Z}x\mathbf{d}y=0).\]
Note that that $\precapprox$ depends only on the values of $\mathbf{d}$ and $\mathbf{e}$ close to $0$.  More precisely, we show below that $\mathbf{d}\precapprox\mathbf{e}$ is equivalent to
\[\forall\epsilon>0\ \exists\delta>0\ \forall x\in X\ \forall y\in Y\ (x\mathbf{e}y<\delta\ \Rightarrow\ x\mathbf{d}y<\epsilon).\]
In particular, $\approx$ defined by
\[\mathbf{d}\approx\mathbf{e}\qquad\Leftrightarrow\qquad\mathbf{d}\precapprox\mathbf{e}\precapprox\mathbf{d}\]
does indeed extend the usual uniform equivalence relation on metrics.  Indeed, $\precapprox$ plays a similarly fundamental role in applications (e.g. see \cite{BiceVignati2018}).

\begin{prp}\label{uniequiv}
\[\mathbf{d}\precapprox\mathbf{e}\qquad\Leftrightarrow\qquad\forall\epsilon>0\ \exists\delta>0\ (<^\mathbf{e}_\delta\ \subseteq\ <^\mathbf{d}_\epsilon).\]
\end{prp}

\begin{proof} Assume that, for every $\epsilon>0$, we have some $\delta>0$ such that $x\mathbf{e}y<\delta$ implies $x\mathbf{d}y<\epsilon$.  For any $Z\subseteq X\times Y$ with $\inf_{(x,y)\in Z}x\mathbf{e}y=0$, we have $(x,y)\in Z$ with $x\mathbf{e}y<\delta$ so $x\mathbf{d}y<\epsilon$ and hence $\inf_{(x,y)\in Z}x\mathbf{d}y<\epsilon$.  Thus $\inf_{(x,y)\in Z}x\mathbf{d}y=0$, as $\epsilon>0$ was arbitrary, i.e. $\mathbf{d}\precapprox\mathbf{e}$.

Conversely, assume we have some $\epsilon>0$ such that, for all $\delta>0$, there exists some $x\in X$ and $y\in Y$ with $x\mathbf{e}y<\delta$ but $x\mathbf{d}y\geq\epsilon$.  In particular, we have $(x_n,y_n)$ with $x_n\mathbf{e}y_n<1/n$ but $x_n\mathbf{d}y_n\geq\epsilon$.  Thus for
\[Z=\{(x_n,y_n):n\in\mathbb{N}\},\]
we have $\inf_{(x,y)\in Z}x\mathbf{e}y=0$ but $\inf_{(x,y)\in Z}x\mathbf{d}y\geq\epsilon>0$, i.e. $\mathbf{d}\not\precapprox\mathbf{e}$.
\end{proof}

Note for (the characteristic function of) any relation $\sqsubset$ and $r\in(0,\infty)$,
\[<^\sqsubset_r\ =\ \sqsubset.\]
Thus by \autoref{uniequiv}, $\precapprox$ reduces to inclusion $\subseteq$ on $\mathbf{Rel}$ so $\precapprox$ is also a valid extension of the $2$-poset structure from $\mathbf{Rel}$ to $\mathbf{GRel}$.

Also note $\precapprox$ can be expressed in terms of $\tfrac{\mathbf{d}}{\mathbf{e}}\in[0,\infty]^{[0,\infty]}$ defined by
\[\tfrac{\mathbf{d}}{\mathbf{e}}(r)=\sup_{x\mathbf{e}y\leq r}x\mathbf{d}y,\]
i.e. $\tfrac{\mathbf{d}}{\mathbf{e}}$ is the smallest monotone function satisfying
\[\tfrac{\mathbf{d}}{\mathbf{e}}(x\mathbf{e}y)\geq x\mathbf{d}y.\]
Specifically, from \autoref{uniequiv} it follows that
\begin{equation}\label{d/e}
\mathbf{d}\precapprox\mathbf{e}\qquad\Leftrightarrow\qquad\lim_{r\rightarrow0}\tfrac{\mathbf{d}}{\mathbf{e}}(r)=0.
\end{equation}

\section{Balls}\label{B}

Often it will also be convenient to consider the unary functions defined from binary functions by fixing one coordinate.  Specifically, for $x\in X$ and $y\in Y$, define $x\mathbf{d}\in[0,\infty]^Y$ and $\mathbf{d}y\in[0,\infty]^X$ by
\[x\mathbf{d}(y)=x\mathbf{d}y=\mathbf{d}y(x).\]
Again we identify subsets with characteristic functions so, for $\sqsubset\ \subseteq X\times Y$,
\begin{align*}
x\sqsubset\quad&=\quad\{y\in Y:x\sqsubset y\}.\\
\sqsubset y\quad&=\quad\{x\in X:x\sqsubset y\}.
\end{align*}
In particular, we define the open upper and lower $\mathbf{d}$-balls with centre $c$ in $X$ or $Y$ and radius $r$ by
\begin{align}
\label{highballs}c^\bullet_r\quad=\quad c\mathrel{<^\mathbf{d}_r}\quad&=\quad\{y\in Y:c\mathbf{d}y<r\}.\\
\label{lowballs}c_\bullet^r\quad=\hspace{13pt}\mathrel{<^\mathbf{d}_r}c\quad&=\quad\{x\in X:x\mathbf{d}c<r\}.
\end{align}
These characterize $\overline{\mathbf{d}}$ and $\underline{\mathbf{d}}$ as follows (taking $\inf\emptyset=\infty$).

\begin{prp}
\begin{align*}
x\overline{\mathbf{d}}z&=\inf\{\epsilon>0:\forall r\in(0,\infty)\ z^\bullet_r\subseteq x^\bullet_{r+\epsilon}\}.\\
z\underline{\mathbf{d}}y&=\inf\{\epsilon>0:\forall r\in(0,\infty)\ z^r_\bullet\subseteq y^{r+\epsilon}_\bullet\}.
\end{align*}
\end{prp}

\begin{proof}
If $x\overline{\mathbf{d}}z<\epsilon$ then, for any $r\in(0,\infty)$ and $w\in z^\bullet_r$, \autoref{hemiprop} yields $x\mathbf{d}w\leq x\overline{\mathbf{d}}z+z\mathbf{d}w<\epsilon+r$ so $w\in x^\bullet_{r+\epsilon}$, i.e. $z^\bullet_r\subseteq x^\bullet_{r+\epsilon}$.  Conversely, say $\epsilon>0$ and $z^\bullet_r\subseteq x^\bullet_{r+\epsilon}$, for all $r\in(0,\infty)$, and take $w\in X$.  If $x\mathbf{d}w=\infty$ then $(x\mathbf{d}w-z\mathbf{d}w)_+=0<\epsilon$.  Otherwise, for all $r\in(z\mathbf{d}w,\infty)$ we have $w\in z^\bullet_r\subseteq x^\bullet_{r+\epsilon}$ and hence $(x\mathbf{d}w-z\mathbf{d}w)_+<r+\epsilon-z\mathbf{d}w$.  As $r>z\mathbf{d}w$ and $w\in X$ were arbitrary, $x\overline{\mathbf{d}}z\leq\epsilon$.  The $\underline{\mathbf{d}}$ statement follows by duality.
\end{proof}

In particular, for any $\sqsubset\ \subseteq X\times Y$,
\begin{align}
\label{sqsubover}x\mathrel{\overline{\hspace{-.7pt}\sqsubset\hspace{-.7pt}}}z\qquad&\Leftrightarrow\qquad(z\sqsubset)\subseteq(x\sqsubset).\\
\label{sqsubunder}z\mathrel{\underline{\hspace{-.7pt}\sqsubset\hspace{-.7pt}}}y\qquad&\Leftrightarrow\qquad(\sqsubset z)\subseteq(\sqsubset y).
\end{align}
In \cite{Erne1991} before Lemma 3.1, these are called the `upper quasiorder' and `lower quasiorder' of $\sqsubset$ (we say preorder instead of quasiorder).  For example, the upper and lower preorder defined from the strict ordering $<$ on $[0,\infty]$ both coincide with the usual ordering on $[0,\infty]$, which we continue to denote by $\leq$ as usual.  More generally, if $X$ is a domain with way-below relation $\ll$ then $\underline{\ll}$ gives back the original ordering on $X$.  From this dual point of view, the lower preorder defined from a transitive relation is just as important as the way-below relation defined from a partial order.  Our thesis is that the same is true for non-symmetric distances as well.

Let $\mathbf{d}^\bullet$ denote the topology on $Y$ generated (as arbitrary unions of finite intersections) by open upper $\mathbf{d}$-balls with centres in $X$, i.e.
\[\mathbf{d}^\bullet\text{ is the topology on $Y$ with subbasis }(x^\bullet_r)_{x\in X,r\in(0,\infty)}.\]
As $x^\bullet_\infty=\bigcup_{r\in(0,\infty)}x^\bullet_r$ and $x^\bullet_0=\emptyset$ are both $\mathbf{d}^\bullet$-open anyway, we could actually take $r\in[0,\infty]$.  Likewise, we let $\mathbf{d}_\bullet$ denote the topology on $X$ generated by open lower $\mathbf{d}$-balls with centres in $Y$, i.e.
\[\mathbf{d}_\bullet\text{ is the topology on $X$ with subbasis }(y_\bullet^r)_{y\in Y,r\in(0,\infty)}.\]
We discuss these further in \autoref{Topology}.  For the moment we just note that $\leq^{\underline{\mathbf{d}}}$ and $\leq^{\overline{\mathbf{d}}}$ are the specialization preorders coming from the $\mathbf{d}^\bullet$ and $\mathbf{d}_\bullet$ topologies.

\begin{prp}\label{specialization}
\begin{align}
\label{specpre}z\leq^{\underline{\mathbf{d}}}y\qquad&\Leftrightarrow\qquad z\text{ is in the $\mathbf{d}^\bullet$-closure of }\{y\}.\\
\label{specpre2}x\leq^{\overline{\mathbf{d}}}z\qquad&\Leftrightarrow\qquad x\text{ is in the $\mathbf{d}_\bullet$-closure of }\{z\}.
\end{align}
\end{prp}

\begin{proof}
Note $z\underline{\mathbf{d}}y=0$ means $x\mathbf{d}y\leq x\mathbf{d}z$, for all $x\in X$, which is equivalent to saying every open upper $\mathbf{d}$-ball containing $z$ must also contain $y$.  Thus the same is true of intersections of such balls and hence unions of such intersections, i.e. all $\mathbf{d}^\bullet$-open sets.  This proves \eqref{specpre} and \eqref{specpre2} again follows by duality.
\end{proof}

\section{The Strict Order}\label{TSO}

Here we examine a strict version $<^\mathbf{d}$ of $\leq^\mathbf{d}$ satisfying an analog of $\mathbf{d}=\overline{\mathbf{d}}\circ\mathbf{d}=\mathbf{d}\circ\underline{\mathbf{d}}$ from \autoref{hemiprop}.  First, consider the following.

\begin{prp}\label{strictmotivation}
For any $r,s\in(0,\infty)$ with $r<s$,
\[x<^\mathbf{d}_ry\quad\Rightarrow\quad(x\leq^\mathbf{d}_r)\text{ is a $\underline{\mathbf{d}}^\bullet$-neighbourhood of }y\quad\Rightarrow\quad x<^\mathbf{d}_sy.\]
\end{prp}

\begin{proof}
If $x<^\mathbf{d}_ry$ then $\delta=r-x\mathbf{d}y>0$ and, for any $z\in(y<^{\underline{\mathbf{d}}}_\delta)$, $\mathbf{d}=\mathbf{d}\circ\underline{\mathbf{d}}$ from \autoref{hemiprop} yields
\[x\mathbf{d}z\leq x\mathbf{d}y+y\underline{\mathbf{d}}z<x\mathbf{d}y+\delta=r.\]
Thus $y\in(y<^{\underline{\mathbf{d}}}_\delta)\subseteq(x\leq^\mathbf{d}_r)$ so $(x\leq^\mathbf{d}_r)$ is a $\underline{\mathbf{d}}^\bullet$-neighbourhood of $y$.

On the other hand, if $(x\leq^\mathbf{d}_r)$ is a $\underline{\mathbf{d}}^\bullet$-neighbourhood of $y$ then, in particular, $x\leq^\mathbf{d}_ry$ so $x\mathbf{d}y\leq r<s$, i.e. $x<^\mathbf{d}_sy$.
\end{proof}

\autoref{strictmotivation} motivates the following definition of $<^\mathbf{d}$.
\begin{equation}\label{Twaybelow}
x<^\mathbf{d}y\qquad\Leftrightarrow\qquad(x\leq^\mathbf{d})\text{ is a $\underline{\mathbf{d}}^\bullet$-neighbourhood of }y.
\end{equation}
As $\underline{\mathbf{d}}$ is a hemimetric, $x<^\mathbf{d}y$ is equivalent to saying there is some open upper $\underline{\mathbf{d}}$-ball with centre $y$ which is entirely $\mathbf{d}$-above $x$, i.e.
\begin{align}
\label{<^d}x<^\mathbf{d}y\qquad&\Leftrightarrow\qquad\exists\epsilon>0\ (y<^{\underline{\mathbf{d}}}_\epsilon)\subseteq(x\leq^\mathbf{d}).\\
\nonumber&\Leftrightarrow\qquad\exists\epsilon>0\ \forall z\in X\ (y<^{\underline{\mathbf{d}}}_\epsilon z\ \Rightarrow\ x\leq^\mathbf{d}z).
\end{align}

When $\mathbf{d}$ itself is a hemimetric, \autoref{hemiprop} yields $\mathbf{d}=\underline{\mathbf{d}}$ so \eqref{specpre} and \eqref{Twaybelow} show that $<^\mathbf{d}$ is the $\mathbf{d}^\bullet$-topological way-below relation $\prec\!\prec$ familiar from Ern\'{e}'s c-spaces \textendash\, see \cite[\S2.5]{Keimel2016}.  In fact, if one considers the motivating example from \cite{Keimel2016}, namely $(C_0(X)_+,\prec\!\prec)$ where $f\prec\!\prec g$ means $f\leq(g-\epsilon)_+$, for some $\epsilon>0$, then again we see that $\prec\!\prec$ is just $<^\mathbf{d}$ for the hemimetric $f\mathbf{d}g=\sup_{x\in X}(f(x)-g(x))_+$.

\begin{prp}\label{<d}
\[\leq^\mathbf{d}\quad\supseteq\quad<^\mathbf{d}\quad=\quad\leq^{\overline{\mathbf{d}}}\circ<^\mathbf{d}\quad=\quad<^\mathbf{d}\circ\leq^{\underline{\mathbf{d}}}\quad\supseteq\quad\leq^\mathbf{d}\circ<^{\underline{\mathbf{d}}}.\]
\end{prp}

\begin{proof}
Note $y<^{\underline{\mathbf{d}}}_\epsilon y$, for all $\epsilon>0$.  So whenever $x<^\mathbf{d}y$, we have $\epsilon>0$ with $y\in(y<^{\underline{\mathbf{d}}}_\epsilon)\subseteq(x\leq^\mathbf{d})$, i.e. $x\leq^\mathbf{d}y$ and hence $\leq^\mathbf{d}\ \supseteq\ <^\mathbf{d}$.

If $x\leq^{\overline{\mathbf{d}}}y<^\mathbf{d}z$ then, for some $\epsilon>0$, $(z<^{\underline{\mathbf{d}}}_\epsilon)\subseteq(y\leq^\mathbf{d})\subseteq(x\leq^\mathbf{d})$, as $\mathbf{d}=\overline{\mathbf{d}}\circ\mathbf{d}$, by \autoref{hemiprop}, so $y\mathbf{d}w=0$ implies $x\mathbf{d}w\leq x\overline{\mathbf{d}}y+y\mathbf{d}w=0$.  Thus $x<^\mathbf{d}z$ and hence $\leq^{\overline{\mathbf{d}}}\circ<^\mathbf{d}\ \subseteq\ <^\mathbf{d}$.

If $x<^\mathbf{d}y\leq^{\underline{\mathbf{d}}}z$ then, for some $\epsilon>0$, $(z<^{\underline{\mathbf{d}}}_\epsilon)\subseteq(y<^{\underline{\mathbf{d}}}_\epsilon)\subseteq(x\leq^\mathbf{d})$, as $\underline{\mathbf{d}}$ is a distance, by \autoref{hemiprop}, so $z\underline{\mathbf{d}}w<\epsilon$ implies $y\underline{\mathbf{d}}w\leq y\underline{\mathbf{d}}z+z\underline{\mathbf{d}}w<\epsilon$.  Thus $x<^\mathbf{d}z$ and hence $<^\mathbf{d}\circ\leq^{\underline{\mathbf{d}}}\ \subseteq\ <^\mathbf{d}$.

The reverse inclusions are immediate from the fact $\leq^{\overline{\mathbf{d}}}$ and $\leq^{\underline{\mathbf{d}}}$ are reflexive, by \autoref{hemiprop}.

If $x\leq^\mathbf{d}y<^{\underline{\mathbf{d}}}z$ then, for some $\epsilon>0$, $(z<^{\underline{\mathbf{d}}}_\epsilon)\subseteq(y\leq^{\underline{\mathbf{d}}})\subseteq(x\leq^\mathbf{d})$, as $\mathbf{d}=\mathbf{d}\circ\underline{\mathbf{d}}$, by \autoref{hemiprop}, so $y\underline{\mathbf{d}}w=0$ implies $x\mathbf{d}w\leq x\mathbf{d}y+y\underline{\mathbf{d}}w=0$.  Thus $x<^\mathbf{d}z$ and hence $\leq^\mathbf{d}\circ<^{\underline{\mathbf{d}}}\ \subseteq\ <^\mathbf{d}$.
\end{proof}

\begin{cor}\label{<ddist}
If $\mathbf{d}\in[0,\infty]^{X\times X}$ is a distance, $<^\mathbf{d}$ is transitive and
\begin{align}
\label{<<=}<^\mathbf{d}\circ\leq^\mathbf{d}\ &\subseteq\ <^\mathbf{d}.\\
\label{<=<}\leq^\mathbf{d}\circ<^\mathbf{d}\ &\subseteq\ <^\mathbf{d}.
\end{align}
\end{cor}

\begin{proof}
As $\mathbf{d}$ is a distance, $\overline{\mathbf{d}},\underline{\mathbf{d}}\leq\mathbf{d}$, by \autoref{hemiprop}, so $\leq^{\overline{\mathbf{d}}},\leq^{\underline{\mathbf{d}}}\ \supseteq\ \leq^\mathbf{d}$ and $<^{\underline{\mathbf{d}}}\ \subseteq\ <^\mathbf{d}$.  Thus, by \autoref{<d},
\begin{align*}
<^\mathbf{d}\circ<^\mathbf{d}\ \subseteq\ <^\mathbf{d}\circ\leq^\mathbf{d}\ &\subseteq\ <^\mathbf{d}\circ\leq^{\underline{\mathbf{d}}}\ =\ <^\mathbf{d}.\\
\subseteq\ \leq^\mathbf{d}\circ<^\mathbf{d}\ &\subseteq\ \leq^{\overline{\mathbf{d}}}\circ<^\mathbf{d}\ =\ <^\mathbf{d}.\\
&\subseteq\ \leq^\mathbf{d}\circ<^{\underline{\mathbf{d}}}\ \subseteq\ <^\mathbf{d}.\qedhere
\end{align*}
\end{proof}

By \autoref{<d}, $<^\mathbf{d}\ =\ <^\mathbf{d}\circ\leq^{\underline{\mathbf{d}}}$ and hence, by \autoref{Kan},
\begin{equation}\label{lequnder}
\leq^{\underline{\mathbf{d}}}\ \ \subseteq\ \underline{<^\mathbf{d}}.
\end{equation}
We can improve this to an equality under a certain interpolation condition.

\begin{prp}\label{under<d}
If $\ \overline{\mathbf{d}}\ \circ<^\mathbf{d}\ \leq\mathbf{d}\ $ then $\ \leq^{\underline{\mathbf{d}}}\ =\underline{<^\mathbf{d}}$.
\end{prp}

\begin{proof}
To prove $\underline{<^\mathbf{d}}\subseteq\ \leq^{\underline{\mathbf{d}}}$, say $z\not\leq^{\underline{\mathbf{d}}}y$, i.e. $z\underline{\mathbf{d}}y>0$, so we have $x\in X$ with $x\mathbf{d}y-x\mathbf{d}z>0$, i.e. $x\mathbf{d}z<x\mathbf{d}y$.  As $\overline{\mathbf{d}}\ \circ<^\mathbf{d}\ \leq\mathbf{d}$, we have $w\in X$ with $x\overline{\mathbf{d}}w<x\mathbf{d}y$ and $w<^\mathbf{d}z$.  Thus $0<x\mathbf{d}y-x\overline{\mathbf{d}}w\leq w\mathbf{d}y$, as $\mathbf{d}=\mathbf{d}\circ\underline{\mathbf{d}}$ by \autoref{hemiprop}, i.e. $w\nleq^\mathbf{d}y$ and hence $w\not<^\mathbf{d}y$, so $z\not\!\!\underline{<^\mathbf{d}}\,y$.
\end{proof}

The similar condition $\mathbf{d}\circ\mathbin{<^{\underline{\mathbf{d}}}}\,\leq\,\mathbf{d}$ can be derived from another interpolation condition involving $\mathbf{d}\mathcal{P}$.  Specifically, for $\mathbf{d}\in[0,\infty]^{X\times Y}$, define $\mathbf{d}\mathcal{P}$ on $X\times\mathcal{P}(Y)$, where $\mathcal{P}(Y)=\{Z:Z\subseteq Y\}$, by
\[x(\mathbf{d}\mathcal{P})Z=\sup_{z\in Z}x\mathbf{d}z.\]
In particular, note $x\leq^{\mathbf{d}\mathcal{P}}Z$ means $x\leq^\mathbf{d}z$, for all $z\in Z$.  Also consider the following condition on closed upper balls $\overline{x}^\bullet_r=\{y\in Y:x\mathbf{d}y\leq r\}$ with finite radius $r<\infty$.
\[\label{minima}\tag{$*$}\text{Every finite radius closed upper $\mathbf{d}$-ball has a $\leq^{\underline{\mathbf{d}}}$-minimum.}\]

\begin{prp}\label{minimaProp}
\[\eqref{minima}\quad\Rightarrow\quad\mathbf{d}\circ\mathbin{\leq^{\underline{\mathbf{d}}\mathcal{P}}}\,\leq\,\mathbf{d}\mathcal{P}\quad\Rightarrow\quad\mathbf{d}\circ\mathbin{<^{\underline{\mathbf{d}}}}\,\leq\,\mathbf{d}.\]
\end{prp}

\begin{proof}
Take any $x\in X$ and $Z\subseteq Y$ and let $r=x(\mathbf{d}\mathcal{P})Z=\sup_{z\in Z}x\mathbf{d}z$.  If $r=\infty$ then we immediately have $x(\mathbf{d}\circ\mathbin{\leq^{\underline{\mathbf{d}}\mathcal{P}}})Y\leq r$.  Otherwise, we have a $\leq^{\underline{\mathbf{d}}}$\hspace{1pt}-minimum $y$ of $\overline{x}^\bullet_r$.  Thus $x\mathbf{d}y\leq r$ and $y\leq^{\underline{\mathbf{d}}}z$, for all $z\in Z$, i.e. $y\leq^{\underline{\mathbf{d}}\mathcal{P}}Z$.  So $x(\mathbf{d}\circ\mathbin{\leq^{\underline{\mathbf{d}}\mathcal{P}}})Z\leq r=x(\mathbf{d}\mathcal{P})Z$, proving the first $\Rightarrow$.

For the second $\Rightarrow$, assume $\mathbf{d}\circ\mathbin{\leq^{\underline{\mathbf{d}}\mathcal{P}}}\,\leq\,\mathbf{d}\mathcal{P}$ and say $x\mathbf{d}y<r$.  Take $\delta$ with $0<\delta<r-x\mathbf{d}y$ so that, as $\mathbf{d}=\mathbf{d}\circ\underline{\mathbf{d}}$, by \autoref{hemiprop},
\[x(\mathbf{d}\mathcal{P})(y<^{\underline{\mathbf{d}}}_\delta)\leq x\mathbf{d}y+y(\underline{\mathbf{d}}\mathcal{P})(y<^{\underline{\mathbf{d}}}_\delta)\leq x\mathbf{d}y+\delta<r.\]
  So we have $z\in Y$ with $x\mathbf{d}z<r$ and $z\leq^{\underline{\mathbf{d}}\mathcal{P}}(y<^{\underline{\mathbf{d}}}_\delta)$ and hence $z<^{\underline{\mathbf{d}}}y$.  Thus $x(\mathbf{d}\ \circ<^{\underline{\mathbf{d}}})y<r$.  As $r>x\mathbf{d}y$ was arbitrary, $\mathbf{d}\circ\mathbin{<^{\underline{\mathbf{d}}}}\,\leq\,\mathbf{d}$.
\end{proof}

When $\mathbf{d}$ is a hemimetric, we can even weaken $\leq$ to $\precapprox$.

\begin{prp}\label{hemi<d}
If $\mathbf{d}$ is a hemimetric then
\[\mathbf{d}\circ\mathbin{\leq^{\mathbf{d}\mathcal{P}}}\,\precapprox\,\mathbf{d}\mathcal{P}\quad\Rightarrow\quad\mathbf{d}\circ\mathbin{<^\mathbf{d}}\leq\,\mathbf{d}.\]
\end{prp}

\begin{proof}
Assume $\mathbf{d}\circ\mathbin{\leq^{\mathbf{d}\mathcal{P}}}\precapprox\,\mathbf{d}\mathcal{P}$ and say $x\mathbf{d}y<r$.  By \autoref{uniequiv}, we have some $\delta>0$ such that $y(\mathbf{d}\mathcal{P})Z\leq\delta$ implies $y(\mathbf{d}\circ\mathbin{\leq^{\mathbf{d}\mathcal{P}}})Z<r-x\mathbf{d}y$.  In particular, we can take $Z=y^\bullet_\delta$ and then we have $z\in X$ with $y\mathbf{d}z<r-x\mathbf{d}y$ and $z\leq^{\mathbf{d}\mathcal{P}}y^\bullet_\delta$.  Thus $x\mathbf{d}z\leq x\mathbf{d}y+y\mathbf{d}z<r$ and $(y<^{\underline{\mathbf{d}}}_\delta)=(y<^\mathbf{d}_\delta)\subseteq(z\leq^\mathbf{d})$, i.e. $z<^\mathbf{d}y$.  As $r>x\mathbf{d}y$ was arbitrary, $\mathbf{d}\circ\mathbin{<^\mathbf{d}}\leq\,\mathbf{d}$.
\end{proof}

\begin{cor}\label{<dunder}
If $\ \overline{\mathbf{d}}\ \circ\leq^\mathbf{d}\ \leq\,\mathbf{d}\ $ and $\ \underline{\mathbf{d}}\circ\mathbin{\leq^{\underline{\mathbf{d}}\mathcal{P}}}\,\precapprox\,\underline{\mathbf{d}}\mathcal{P}\ $ then $\ \leq^{\underline{\mathbf{d}}}\ =\underline{<^\mathbf{d}}$.
\end{cor}

\begin{proof}  By \autoref{<d}, \autoref{hemiprop} and \autoref{hemi<d} (for $\underline{\mathbf{d}}$),
\[(\overline{\mathbf{d}}\ \circ<^\mathbf{d})\leq(\overline{\mathbf{d}}\ \circ\leq^\mathbf{d}\circ<^{\underline{\mathbf{d}}})\leq(\mathbf{d}\ \circ<^{\underline{\mathbf{d}}})=(\mathbf{d}\circ\underline{\mathbf{d}}\ \circ<^{\underline{\mathbf{d}}})\leq(\mathbf{d}\circ\underline{\mathbf{d}})=\mathbf{d}.\]
Thus $\leq^{\underline{\mathbf{d}}}\ =\underline{<^\mathbf{d}}$, by \autoref{under<d}.
\end{proof}

For example, \eqref{minima} and hence $\leq^\mathbf{d}\ =\underline{<^\mathbf{d}}$ holds in $C_0(X)_+$, where again $f\mathbf{d}g=\sup_{x\in X}(f(x)-g(x))_+$.  Indeed, for any $f\in C_0(X)_+$ and $r\in[0,\infty]$, we see that $(f-r)_+$ is the $\leq^\mathbf{d}$-minimum of the closed upper $\mathbf{d}$-ball $\overline{f}_r^\bullet$ with centre $f$ and radius $r$.

But if we consider the opposite hemimetric on $C_0(X)$ given by $f\mathbf{e}g=\sup_{x\in X}(g(x)-f(x))_+$ and $X$ is not compact then $<^\mathbf{e}$ is vacuous, owing the fact any $f,g\in C_0(X)_+$ must vanish at infinity.  This means $\underline{<^\mathbf{e}}$ is trivial, i.e. $f\underline{<^\mathbf{e}}\,g$ for arbitrary $f,g\in C_0(X)_+$.  On the other hand, here $\leq^{\underline{\mathbf{e}}}\ =\ \leq^\mathbf{e}$ is just the opposite of the pointwise ordering on $C_0(X)_+$.  In particular, $\leq^{\underline{\mathbf{e}}}$ is not trivial, so the inclusion in \eqref{lequnder} is strict.

Also $\overline{\mathbf{d}}\ \circ<^\mathbf{d}\ \leq\mathbf{d}$ and hence $\leq^{\underline{\mathbf{d}}}\ =\underline{<^\mathbf{d}}$ holds in spaces formal balls, which will be crucial in our future work when we look at generalized (pre)domains.

\section{\bf Nets}\label{N}

We consider nets in a slightly more general sense than usual.  Specifically, as we deal with non-hemimetric distances, we must also deal with non-reflexive nets (to allow for $\mathbf{d}$-Cauchy nets even when $\leq^\mathbf{d}$ is not reflexive).  So by a net we mean a non-empty set indexed by a directed set $\Lambda$, i.e. we have (possibly non-reflexive) transitive $\mathbin{\prec}\subseteq\Lambda\times\Lambda$ satisfying
\[\forall\gamma,\delta\ \exists\lambda\ (\gamma,\delta\prec\lambda).\]

As usual, we define limits by
\begin{equation}\label{convergencedef}
x_\lambda\rightarrow x\quad\Leftrightarrow\quad\forall\text{ open }O\ni x\ \exists\gamma\in\Lambda\ (x_\lambda)_{\lambda\succ\gamma}\subseteq O
\end{equation}
In fact, these are also the limits with respect to the preorder $\preceq$ given by
\[x\preceq y\qquad\Leftrightarrow\qquad(\prec x)\subseteq(\prec y)\]
as in \eqref{sqsubunder}.  Also note that is suffices to verify \eqref{convergencedef} for all open $O$ in a subbasis $\mathcal{S}$ for the topology.  Indeed, as nets are indexed by directed sets, if \eqref{convergencedef} holds for all $O\in\mathcal{S}$ then \eqref{convergencedef} holds for all finite intersections of elements of $\mathcal{S}$ and hence for all unions of finite intersections of elements of $\mathcal{S}$, i.e. all open sets.  In particular, for any topologies $\mathcal{T}$ and $\mathcal{U}$, convergence in their supremum $\mathcal{T}\vee\mathcal{U}(=$ the topology with subbasis $\mathcal{T}\cup\mathcal{U}$) is the same as convergence in both $\mathcal{T}$ and $\mathcal{U}$, i.e.
\begin{equation}\label{joinconverge}
x_\lambda\xrightarrow{\mathcal{T}\vee\mathcal{U}}x\qquad\Leftrightarrow\qquad x_\lambda\xrightarrow{\mathcal{T}}x\quad\text{and}\quad x_\lambda\xrightarrow{\mathcal{U}}x.
\end{equation}

Limits in $[-\infty,\infty]$ are considered with respect to the usual interval topology and limits inferior and superior are defined as usual by
\begin{align*}
\liminf_\lambda r_\lambda&=\lim_\gamma\inf_{\gamma\prec\lambda}r_\lambda.\\
\limsup_\lambda r_\lambda&=\lim_\gamma\sup_{\gamma\prec\lambda}r_\lambda.
\end{align*}
Note limits inferior/superior are below/above infima/suprema, i.e.
\begin{equation}\label{liminflimsup}
\inf_\lambda r_\lambda\leq\liminf_\lambda r_\lambda\leq\limsup_\lambda r_\lambda\leq\sup_\lambda r_\lambda.
\end{equation}
Also, $(r_\lambda)$ converges in $[-\infty,\infty]$ iff
\begin{equation}\label{convergence}
\limsup_\lambda r_\lambda\leq\liminf_\lambda r_\lambda,
\end{equation}
in which case $\lim_\lambda r_\lambda=\limsup_\lambda r_\lambda=\liminf_\lambda r_\lambda$.

We also use a number of standard facts like
\begin{align}
\label{liminfinf}\liminf_\lambda(r_\lambda+s_\lambda)&\geq\liminf_\lambda r_\lambda+\liminf_\lambda s_\lambda.\\
\label{liminfsup}\liminf_\lambda(r_\lambda+s_\lambda)&\leq\liminf_\lambda r_\lambda+\limsup_\lambda s_\lambda\leq\limsup_\lambda(r_\lambda+s_\lambda).\\
\label{limsupsup}&\quad\limsup_\lambda r_\lambda+\limsup_\lambda s_\lambda\geq\limsup_\lambda(r_\lambda+s_\lambda).
\end{align}
Note these are only valid when we do not end up with $\infty-\infty$ in the middle, which is not a problem on $[-t,\infty]$, for any $t\in[0,\infty)$.  Also,
\[\liminf_\lambda(-r_\lambda)=-\limsup_\lambda r_\lambda.\]
Indeed, in the finite case this follows from \eqref{liminfsup} by taking $s_\lambda=-r_\lambda$, while the infinite case can be verified directly.  Also, as $r\mapsto r_+$ is continuous and (non-strictly) increasing on $[-\infty,\infty]$, we have
\[\liminf_\lambda(r_{\lambda+})=(\liminf_\lambda r_\lambda)_+\qquad\text{and}\qquad\limsup_\lambda(r_{\lambda+})=(\limsup_\lambda r_\lambda)_+.\]
For example, combining these facts yields
\begin{equation}\label{s-r}
\limsup_\lambda((s-r_\lambda)_+)=(\limsup_\lambda(s-r_\lambda))_+=(s-\liminf_\lambda r_\lambda)_+,
\end{equation}
as long as $s$ or $\liminf_\lambda r_\lambda$ is finite.

Let us adopt the convention that when nets are written on the left of $\mathbf{d}$ we take the limit superior, while on the right we take the limit inferior:
\begin{align*}
(x_\lambda)\mathbf{d}x&=\limsup_\lambda x_\lambda\mathbf{d}x.\\
x\mathbf{d}(x_\lambda)&=\liminf_\gamma x\mathbf{d}x_\lambda.
\end{align*}
We also extend this notation to unary functions, defining
\begin{align*}
(x_\lambda)\mathbf{d}&=\limsup_\lambda x_\lambda\mathbf{d}.\\
\mathbf{d}(x_\lambda)&=\liminf_\gamma\mathbf{d}x_\lambda.
\end{align*}
(The limits here are pointwise, i.e. in the product topology of $[0,\infty]^X$).

To avoid repetition, from now on we assume $X=Y$, i.e.
\[\textbf{We are given a set $X$ and functions }\mathbf{d},\mathbf{e}\in[0,\infty]^{X\times X}.\]

\begin{prp}\label{trinet}
For any $(z_\lambda)\subseteq X$,
\[x(\mathbf{d}\circ\mathbf{e})y\leq x\mathbf{d}(z_\lambda)+(z_\lambda)\mathbf{e}y.\]
\end{prp}

\begin{proof}
By \eqref{liminfsup},
\[\inf_{z\in Z}(x\mathbf{d}z+z\mathbf{e}y)\leq\liminf_\lambda(x\mathbf{d}z_\lambda+z_\lambda\mathbf{e}y)\leq\liminf_\lambda(x\mathbf{d}z_\lambda)+\limsup_\lambda(z_\lambda\mathbf{e}y).\]
\end{proof}

\section{\bf Cauchy Nets}\label{CN}

\begin{dfn}\label{CNdfn}
For any net $(x_\lambda)\subseteq X$, define
\begin{align}
\label{pre-Cauchy}\lim_\gamma\limsup_\delta x_\gamma\mathbf{d}x_\delta=0\quad&\Leftrightarrow\quad(x_\lambda)\text{ is \emph{$\mathbf{d}$-pre-Cauchy}}.\\
\label{Cauchy}\lim_\gamma\sup_{\gamma\prec\delta} x_\gamma\mathbf{d}x_\delta=0\quad&\Leftrightarrow\quad(x_\lambda)\text{ is \emph{$\mathbf{d}$-Cauchy}}.
\end{align}
\end{dfn}

Equivalently, $(x_\lambda)$ is $\mathbf{d}$-Cauchy if and only if
\[\lim_{\gamma\prec\delta}x_\gamma\mathbf{d}x_\delta=0,\]
when we consider $\prec$ itself as a directed subset of $\Lambda\times\Lambda$ with respect to the product ordering $\prec\times\prec$.  These nets are `increasing modulo $\epsilon$', in a certain sense.  More precisely, they can be characterized by $<^\mathbf{d}_\epsilon$:
\begin{align*}
\forall\epsilon>0\ \exists\gamma_0\ \forall\gamma\succ\gamma_0\ \exists\delta_0\ \forall\delta\succ\delta_0\ (x_\gamma<^\mathbf{d}_\epsilon x_\delta)\quad&\Leftrightarrow\quad(x_\lambda)\text{ is \emph{$\mathbf{d}$-pre-Cauchy}}.\\
\forall\epsilon>0\ \exists\gamma_0\ \forall\gamma\succ\gamma_0\ \hspace{18pt}\forall\delta\succ\gamma\hspace{6pt}(x_\gamma<^\mathbf{d}_\epsilon x_\delta)\quad&\Leftrightarrow\quad(x_\lambda)\text{ is \emph{$\mathbf{d}$-Cauchy}}.
\end{align*}

In particular, if $\sqsubset$ is a transitive relation then the $\sqsubset$-Cauchy nets are precisely the increasing nets, at least beyond a certain point $\gamma_0$.  On the other hand, the $\sqsubset$-pre-Cauchy nets are the `directed nets' from \cite[Definition O-1.2]{GierzHofmannKeimelLawsonMisloveScott2003}.  In the literature on hemimetrics, $\mathbf{d}$-Cauchy nets are more often considered than $\mathbf{d}$-pre-Cauchy nets (a notable exception is \cite{Wagner1997}, where sequences that we would call pre-Cauchy/Cauchy are called Cauchy/strongly Cauchy respectively).  However, most results on $\mathbf{d}$-Cauchy nets can be generalized to $\mathbf{d}$-pre-Cauchy nets without difficulty, as we demonstrate, and these are sometimes more convenient to work with (e.g. it suffices to consider $\mathbf{d}$-pre-Cauchy nets indexed by posets, while with $\mathbf{d}$-Cauchy nets we must consider more general transitive relations).

On the other hand, from a metric space point of view, both \eqref{pre-Cauchy} and \eqref{Cauchy} extend the usual notion of a Cauchy net.  

\begin{prp}\label{symCauchy}
If $\mathbf{d}$ is a symmetric distance, i.e. $\mathbf{d}=\mathbf{d}^\mathrm{op}\leq\mathbf{d}\circ\mathbf{d}$,
\[(x_\lambda)\text{ is $\mathbf{d}$-Cauchy}\qquad\Leftrightarrow\qquad(x_\lambda)\text{ is $\mathbf{d}$-pre-Cauchy}.\]
\end{prp}

\begin{proof}
The $\Rightarrow$ part is immediate.  Conversely, if $(x_\lambda)\subseteq X$ is $\mathbf{d}$-pre-Cauchy then, for every $\epsilon>0$, we have $\alpha,\beta$ such that, for all $\gamma\succ\beta$, $x_\alpha\mathbf{d}x_\gamma<\epsilon$.  Thus, for all $\delta\succ\gamma$,  $\mathbf{d}=\mathbf{d}^\mathrm{op}\leq\mathbf{d}\circ\mathbf{d}$ yields $x_\gamma\mathbf{d}x_\delta\leq x_\alpha\mathbf{d}x_\gamma+x_\alpha\mathbf{d}x_\delta<2\epsilon$, i.e. $(x_\lambda)$ is $\mathbf{d}$-Cauchy.
\end{proof}

Here are a few basic but important facts about pre-Cauchy nets.  Note a version of \eqref{dunder} below appears in \cite[Theorem 2.26]{Wagner1997}.

\begin{thm}\label{Clim}\
\begin{enumerate}
\item\label{preCauchysub} If $(x_\lambda)\subseteq X$ is $\mathbf{d}$-pre-Cauchy then $(x_\lambda)$ has a $\mathbf{d}$-Cauchy subnet.

\item\label{dunder}  If $(x_\lambda)\subseteq X$ is $\overline{\mathbf{d}}$-pre-Cauchy then $x_\lambda\mathbf{d}$ converges (pointwise).

\item\label{dover}  If $(x_\lambda)\subseteq X$ is $\underline{\mathbf{d}}$-pre-Cauchy then $\mathbf{d}x_\lambda$ converges (pointwise) and
\begin{equation}\label{Climeq}
(x_\lambda)\underline{\mathbf{d}}y=\sup_{x\in X}(x\mathbf{d}y-x\mathbf{d}(x_\lambda))_+.
\end{equation}

\item\label{dCdist}  If $(x_\lambda)\subseteq X$ is $\mathbf{d}$-pre-Cauchy and $\mathbf{d}$ is a distance then
\begin{equation}\label{epreC}
\mathbf{d}(x_\lambda)=\overline{\mathbf{d}}(x_\lambda)\qquad\text{and}\qquad(x_\lambda)\mathbf{d}=(x_\lambda)\underline{\mathbf{d}}.
\end{equation}
\end{enumerate}
\end{thm}

\begin{proof}\
\begin{itemize}
\item[\eqref{preCauchysub}]  If $\Lambda$ is finite then it has a maximum $\gamma$, which means the single element net $x_\gamma$ is a $\mathbf{d}$-Cauchy subnet.  Otherwise, let $|F|$ denote the cardinality of $F$ and consider the finite subsets of $\Lambda$
\[\mathcal{F}(\Lambda)=\{F\subseteq\Lambda:|F|<\infty\}\]
directed by $\subsetneqq$.  We define a map $f:\mathcal{F}(\Lambda)\setminus\{\emptyset\}\rightarrow\Lambda$ recursively as follows.  Let $f(\{\lambda\})=\lambda$, for all $\lambda\in\Lambda$.  Given $F\in\mathcal{F}(\Lambda)\setminus\{\emptyset\}$, take $f(F)\in\Lambda$ such that, for all $E\subsetneqq F$, $f(E)\prec f(F)$ and
\[x_{f(E)}\mathbf{d}x_{f(F)}\leq\limsup_\lambda x_{f(E)}\mathbf{d}x_\lambda+2^{-|E|}.\]
In particular, $\lambda\prec f(F)$ whenever $\lambda\in F\neq\{\lambda\}$.  This means that $\{f(F):F\in\mathcal{F}(\Lambda)\setminus\{\emptyset\}\}$ is cofinal in $\Lambda$ and hence $(x_{f(F)})$ is a subnet of $(x_\lambda)$, which yields the second $\leq$ in
\begin{align*}
\limsup_{E\in\mathcal{F}(\Lambda)\setminus\{\emptyset\}}\sup_{E\subsetneqq F}x_{f(E)}\mathbf{d}x_{f(F)}&\leq\limsup_{E\in\mathcal{F}(\Lambda)\setminus\{\emptyset\}}\limsup_\lambda(x_{f(E)}\mathbf{d}x_\lambda+2^{-|E|})\\
&=\limsup_{E\in\mathcal{F}(\Lambda)\setminus\{\emptyset\}}\limsup_\lambda x_{f(E)}\mathbf{d}x_\lambda\\
&\leq\limsup_\gamma\limsup_\lambda x_\gamma\mathbf{d}x_\lambda.\\
&=0
\end{align*}
Thus $(x_{f(F)})$ is a $\mathbf{d}$-Cauchy subnet of $(x_\lambda)$.\\

\item[\eqref{dunder}]  If $(x_\lambda)$ is $\overline{\mathbf{d}}$-pre-Cauchy then, for all $y\in X$,
\begin{align*}
&\quad\limsup_\lambda x_\lambda\mathbf{d}y\\
&\leq\limsup_\lambda\inf_{z\in X}(x_\lambda\overline{\mathbf{d}}z+z\mathbf{d}y)\quad\text{as $\mathbf{d}=\overline{\mathbf{d}}\circ\mathbf{d}$, by \autoref{hemiprop},}\\
&\leq\limsup_\lambda\inf_\gamma(x_\lambda\overline{\mathbf{d}}x_\gamma+x_\gamma\mathbf{d}y)\\
&\leq\limsup_\lambda\liminf_\gamma(x_\lambda\overline{\mathbf{d}}x_\gamma+x_\gamma\mathbf{d}y)\quad\text{by \eqref{liminflimsup}}\\
&\leq\limsup_\lambda(\limsup_\gamma x_\lambda\overline{\mathbf{d}}x_\gamma+\liminf_\gamma x_\gamma\mathbf{d}y)\quad\text{by \eqref{liminfsup}}\\
&=\limsup_\lambda\limsup_\gamma x_\lambda\overline{\mathbf{d}}x_\gamma+\liminf_\gamma x_\gamma\mathbf{d}y\\
&=\liminf_\gamma x_\gamma\mathbf{d}y\quad\text{as $(x_\lambda)$ is $\overline{\mathbf{d}}$-pre-Cauchy}.
\end{align*}
Thus $x_\lambda\mathbf{d}y$ converges, by \eqref{convergence}.\\

\item[\eqref{dover}] If $(x_\lambda)$ is $\underline{\mathbf{d}}$-pre-Cauchy then, for all $y\in X$,
\begin{align*}
&\quad\limsup_\lambda y\mathbf{d}x_\lambda\\
&\leq\inf_{z\in X}\limsup_\lambda(y\mathbf{d}z+z\underline{\mathbf{d}}x_\lambda)\quad\text{as $\mathbf{d}=\mathbf{d}\circ\underline{\mathbf{d}}$, by \autoref{hemiprop},}\\
&\leq\inf_\gamma\limsup_\lambda(y\mathbf{d}x_\gamma+x_\gamma\underline{\mathbf{d}}x_\lambda)\\
&\leq\liminf_\gamma\limsup_\lambda(y\mathbf{d}x_\gamma+x_\gamma\underline{\mathbf{d}}x_\lambda)\quad\text{by \eqref{liminflimsup}}\\
&=\liminf_\gamma(y\mathbf{d}x_\gamma+\limsup_\lambda x_\gamma\underline{\mathbf{d}}x_\lambda)\\
&\leq\liminf_\gamma y\mathbf{d}x_\gamma+\limsup_\gamma\limsup_\lambda x_\gamma\underline{\mathbf{d}}x_\lambda\quad\text{by \eqref{liminfsup}}\\
&=\liminf_\gamma y\mathbf{d}x_\gamma\quad\text{as $(x_\lambda)$ is $\underline{\mathbf{d}}$-pre-Cauchy}.
\end{align*}
Thus $y\mathbf{d}x_\lambda$ converges, by \eqref{convergence}.\\

\item[\eqref{Climeq}]  First note that
\begin{align*}
\sup_{z\in X}(z\mathbf{d}y-z\mathbf{d}(x_\lambda))_+&=\sup_{z\in X}(z\mathbf{d}y-\liminf_\lambda z\mathbf{d}x_\lambda)_+\\
&\leq\sup_{z\in X}\limsup_\lambda(z\mathbf{d}y-z\mathbf{d}x_\lambda)_+\quad\text{by \eqref{s-r}}\\
\intertext{(if $\liminf_\lambda z\mathbf{d}x_\lambda<\infty$, otherwise $(z\mathbf{d}y-\liminf_\lambda z\mathbf{d}x_\lambda)_+=0$)}
&\leq\sup_{z\in X}\limsup_\lambda x_\lambda\underline{\mathbf{d}}y\quad\text{by \eqref{dunderdef}}\\
&=\limsup_\lambda x_\lambda\underline{\mathbf{d}}y\\
&=(x_\lambda)\underline{\mathbf{d}}y.
\end{align*}
For the converse, take $\epsilon\in(0,\infty)$ and replace the $\underline{\mathbf{d}}$-pre-Cauchy net $(x_\lambda)$ with a subnet if necessary so that, for all $\gamma$,
\[\limsup_\lambda x_\gamma\underline{\mathbf{d}}x_\lambda<\epsilon.\]
Note this suffices to prove the result for the original net as we already know that $x\mathbf{d}x_\lambda$ converges, by \eqref{dover}, and $x_\lambda\underline{\mathbf{d}}y$ converges, by \eqref{dunder} (note applying $\overline{\phantom{\mathbf{d}}}$ to the hemimetric $\underline{\mathbf{d}}$ leaves it unchanged)

We first claim that, for all $z\in Z$,
\begin{equation}\label{infsupclaim}
\limsup_\gamma(z\mathbf{d}y-z\mathbf{d}x_\gamma)_+=(z\mathbf{d}y-\liminf_\gamma z\mathbf{d}x_\gamma)_+.
\end{equation}
If $\liminf_\gamma z\mathbf{d}x_\gamma<\infty$ then this follows from \eqref{s-r}.  If $\liminf_\gamma z\mathbf{d}x_\gamma=\infty$ then, using the fact that $\mathbf{d}=\mathbf{d}\circ\underline{\mathbf{d}}$ by \autoref{hemiprop},
\[\infty=\liminf_\lambda z\mathbf{d}x_\lambda\leq z\mathbf{d}x_\gamma+\liminf_\lambda x_\gamma\underline{\mathbf{d}}x_\lambda<z\mathbf{d}x_\gamma+\epsilon,\]
for all $\gamma$.  Thus $\infty=z\mathbf{d}x_\gamma$, for all $\gamma$, so
\[\limsup_\gamma(z\mathbf{d}y-z\mathbf{d}x_\gamma)_+=0=(z\mathbf{d}y-\liminf_\gamma z\mathbf{d}x_\gamma)_+,\]
again proving the claim.

Now consider
\[(x_\lambda)\underline{\mathbf{d}}y=\limsup_\lambda\sup_{z\in X}(z\mathbf{d}y-z\mathbf{d}x_\lambda)_+.\]
As $z\mathbf{d}x_\gamma\leq z\mathbf{d}x_\lambda+x_\lambda\underline{\mathbf{d}}x_\gamma$, by $\mathbf{d}=\mathbf{d}\circ\underline{\mathbf{d}}$ from \autoref{hemiprop}, it follows that $-z\mathbf{d}x_\lambda\leq-z\mathbf{d}x_\gamma+x_\lambda\underline{\mathbf{d}}x_\gamma$, by \eqref{a<b+c}, and hence
\begin{align*}
(x_\lambda)\underline{\mathbf{d}}y&\leq\limsup_\lambda\sup_{z\in X}\inf_\gamma(z\mathbf{d}y+(-z\mathbf{d}x_\gamma+x_\lambda\underline{\mathbf{d}}x_\gamma))_+\\
&\leq\limsup_\lambda\sup_{z\in X}\limsup_\gamma(z\mathbf{d}y+(-z\mathbf{d}x_\gamma+x_\lambda\underline{\mathbf{d}}x_\gamma))_+\quad\text{by \eqref{liminflimsup}}\\
&\leq\limsup_\lambda\sup_{z\in X}\limsup_\gamma((z\mathbf{d}y-z\mathbf{d}x_\gamma)+x_\lambda\underline{\mathbf{d}}x_\gamma)_+\quad\text{as }\limsup_\gamma x_\lambda\underline{\mathbf{d}}x_\gamma<\epsilon\\
\intertext{(and $a+(-b+c)\leq(a-b)+c$ whenever $c\in[0,\infty)$ \textendash\, see \eqref{a+(-b+c)})}
&\leq\limsup_\lambda\sup_{z\in X}\limsup_\gamma((z\mathbf{d}y-z\mathbf{d}x_\gamma)_++x_\lambda\underline{\mathbf{d}}x_\gamma)\quad\text{as }(a+b)_+\leq a_++b_+\\
&\leq\limsup_\lambda\sup_{z\in X}(\limsup_\gamma(z\mathbf{d}y-z\mathbf{d}x_\gamma)_++\limsup_\gamma x_\lambda\underline{\mathbf{d}}x_\gamma)\quad\text{by \eqref{limsupsup}}\\
&=\sup_{z\in X}\limsup_\gamma(z\mathbf{d}y-z\mathbf{d}x_\gamma)_++\limsup_\lambda\limsup_\gamma x_\lambda\underline{\mathbf{d}}x_\gamma\\
&=\sup_{z\in X}\limsup_\gamma(z\mathbf{d}y-z\mathbf{d}x_\gamma)_+\quad\text{as $(x_\lambda)$ is $\underline{\mathbf{d}}$-pre-Cauchy}\\
&=\sup_{z\in X}(z\mathbf{d}y-\liminf_\gamma z\mathbf{d}x_\gamma)_+\quad\text{by \eqref{infsupclaim}}\\
&=\sup_{z\in X}(z\mathbf{d}y-z\mathbf{d}(x_\lambda))_+.
\end{align*}

\item[\eqref{dCdist}]  By \eqref{dis}, $\overline{\mathbf{d}}\leq\mathbf{d}$ so we have $\overline{\mathbf{d}}(x_\lambda)\leq\mathbf{d}(x_\lambda)$.  Conversely,
\begin{align*}
y\mathbf{d}(x_\lambda)&=\liminf_\lambda y\mathbf{d}x_\lambda\\
&\leq\inf_\gamma\liminf_\lambda(y\overline{\mathbf{d}}x_\gamma+x_\gamma\mathbf{d}x_\lambda)\quad\text{as $\mathbf{d}=\overline{\mathbf{d}}\circ\mathbf{d}$ by \autoref{hemiprop}}\\
&\leq\liminf_\gamma\liminf_\lambda(y\overline{\mathbf{d}}x_\gamma+x_\gamma\mathbf{d}x_\lambda)\quad\text{by \eqref{liminflimsup}}\\
&=\liminf_\gamma(y\overline{\mathbf{d}}x_\gamma+\liminf_\lambda x_\gamma\mathbf{d}x_\lambda)\\
&\leq\liminf_\gamma y\overline{\mathbf{d}}x_\gamma+\limsup_\gamma\limsup_\lambda x_\gamma\mathbf{d}x_\lambda\quad\text{by \eqref{liminfsup}}\\
&=\liminf_\gamma y\overline{\mathbf{d}}x_\gamma\quad\text{as $(x_\lambda)$ is $\mathbf{d}$-pre-Cauchy}\\
&=y\overline{\mathbf{d}}(x_\lambda).
\end{align*}
Again by \eqref{dis}, $\underline{\mathbf{d}}\leq\mathbf{d}$ so $(x_\lambda)\underline{\mathbf{d}}\leq(x_\lambda)\mathbf{d}$, while conversely,
\begin{align*}
(x_\lambda)\mathbf{d}y&=\limsup_\gamma x_\gamma\mathbf{d}y\\
&\leq\limsup_\gamma\inf_\lambda(x_\gamma\mathbf{d}x_\lambda+x_\lambda\underline{\mathbf{d}}y)\quad\text{as $\mathbf{d}=\mathbf{d}\circ\underline{\mathbf{d}}$ by \autoref{hemiprop}}\\
&\leq\limsup_\gamma\limsup_\lambda(x_\gamma\mathbf{d}x_\lambda+x_\lambda\underline{\mathbf{d}}y)\quad\text{by \eqref{liminflimsup}}\\
&\leq\limsup_\gamma(\limsup_\lambda x_\gamma\mathbf{d}x_\lambda+\limsup_\lambda x_\lambda\underline{\mathbf{d}}y)\quad\text{by \eqref{limsupsup}}\\
&=\limsup_\gamma\limsup_\lambda x_\gamma\mathbf{d}x_\lambda+\limsup_\lambda x_\lambda\underline{\mathbf{d}}y\\
&=\limsup_\lambda x_\lambda\underline{\mathbf{d}}y\quad\text{as $(x_\lambda)$ is $\mathbf{d}$-pre-Cauchy}\\
&=(x_\lambda)\underline{\mathbf{d}}y.\qedhere
\end{align*}
\end{itemize}
\end{proof}

\section{\bf Holes}\label{Topology}
Define the open upper/lower holes with centre $c\in X$ and radius $r$ by
\begin{align*}
c^\circ_r\quad=\hspace{13pt}\mathrel{>^\mathbf{d}_r}c\quad&=\quad\{x\in X:x\mathbf{d}c>r\}.\\
c_\circ^r\quad=\quad c\mathrel{>^\mathbf{d}_r}\quad&=\quad\{x\in X:c\mathbf{d}x>r\}.
\end{align*}
Note these are defined just like open balls in \eqref{highballs} and \eqref{lowballs} but with $<$ reversed.  Let $\mathbf{d}^\circ$, $\mathbf{d}_\circ$, $\mathbf{d}^\circ_\circ$, $\mathbf{d}^\bullet_\bullet$, $\mathbf{d}^\bullet_\circ$ and $\mathbf{d}^\circ_\bullet$ denote the topologies generated by the corresponding balls and holes, i.e. by arbitrary unions of finite intersections, e.g. $\mathbf{d}_\circ$ is the topology with subbasis $(x^r_\circ)_{x\in X,r\in(0,\infty)}$ and $\mathbf{d}^\bullet_\circ=\mathbf{d}^\bullet\vee\mathbf{d}_\circ$ etc..  As with balls, we could even take $r\in[0,\infty]$, as $x_\circ^0=\bigcup_{r\in(0,\infty)}x_\circ^r$ and $\emptyset=x_\circ^\infty$.  Beware that in general these subbases are not bases \textendash\, for hemimetric $\mathbf{d}$, the balls form a basis for the ball topologies, by \cite[Lemma 6.1.5]{Goubault2013}, but even this can fail for more general distances.

Up until now, most of the literature has focused on ball topologies.  However, as mentioned in \cite[Exercise 6.2.11]{Goubault2013}, hole topologies generalize the upper topology from order theory.  This allows for simple generalizations of certain order theoretic concepts.  Also, the double hole topology $\mathbf{d}^\circ_\circ$ coincides with various kinds of weak topologies, although this too does not appear to be widely recognized.  For example, the double hole topology is the usual product topology on products of bounded intervals, the weak operator topology on projections on a Hilbert space and the Wijsman topology on subsets of $X$ (see \cite[Examples 5 and 6 and \S5.3]{Bice2015v4}).

We denote convergence in $\mathbf{d}^\bullet$, $\mathbf{d}_\circ$, $\mathbf{d}^\bullet_\circ$, etc. by $\barrow$, $\arrowc$, $\barrowc$, etc..

\begin{prp}\label{convchars}
For any net $(x_\lambda)\subseteq X$,
\begin{align}
\label{abdef}x_\lambda\arrowb x\quad&\Leftrightarrow\quad(x_\lambda)\mathbf{d}\leq x\mathbf{d}.\\
\label{acdef}x_\lambda\arrowc x\quad&\Leftrightarrow\quad\mathbf{d}(x_\lambda)\geq\mathbf{d}x.
\end{align}
\end{prp}

\begin{proof}\
\begin{itemize}
\item[\eqref{abdef}]  Recall that for convergence it suffices to consider subbasic open sets, in this case the balls $y^r_\bullet$, for $y\in X$ and $r\in(0,\infty)$.  So $x_\lambda\arrowb x$ means that, for all $y\in X$ and $r\in(0,\infty)$, if $x\in y^r_\bullet$ then $(x_\lambda)_{\lambda\succ\gamma}\subseteq y^r_\bullet$, for some $\gamma$.  Thus if $x\mathbf{d}y<r$ then $\limsup_\lambda x_\lambda\mathbf{d}y\leq r$.  As $r$ and $y$ were arbitrary, this means $\limsup_\lambda x_\lambda\mathbf{d}y\leq x\mathbf{d}y$ and hence $(x_\lambda)\mathbf{d}\leq x\mathbf{d}$.  Conversely, if $(x_\lambda)\mathbf{d}\leq x\mathbf{d}$, i.e. $\limsup_\lambda x_\lambda\mathbf{d}y\leq x\mathbf{d}y$, for all $y\in X$, then $x\mathbf{d}y<r$ implies that $\limsup_\lambda x_\lambda\mathbf{d}y<r$, for all $r\in(0,\infty)$, and hence $x_\lambda\arrowb x$.

\item[\eqref{acdef}]  Likewise $x_\lambda\arrowc x$ means that, for all $y\in X$ and $r\in(0,\infty)$, if $x\in y^r_\circ$ then $(x_\lambda)_{\lambda\succ\gamma}\subseteq y^r_\circ$, for some $\gamma$.  Thus if $y\mathbf{d}x>r$ then $\liminf_\lambda y\mathbf{d}x_\lambda\geq r$.  As $r$ and $y$ were arbitrary, this means $\liminf_\lambda y\mathbf{d}x_\lambda\geq y\mathbf{d}x$ and hence $\mathbf{d}(x_\lambda)\geq\mathbf{d}x$.  Conversely, if $\mathbf{d}(x_\lambda)\geq\mathbf{d}x$, i.e. $\liminf_\lambda y\mathbf{d}x_\lambda\geq y\mathbf{d}x$, for all $y\in X$, then $y\mathbf{d}x>r$ implies that $\liminf_\lambda y\mathbf{d}x_\lambda>r$, for all $r\in(0,\infty)$, and hence $x_\lambda\arrowc x$.\qedhere
\end{itemize}
\end{proof}

Likewise,
\begin{align}
\label{badef}x_\lambda\barrow x\qquad\Leftrightarrow\qquad(x_\lambda)\mathbf{d}^\mathrm{op}\leq x\mathbf{d}^\mathrm{op}\qquad\Leftrightarrow\qquad&\limsup_\lambda\mathbf{d}x_\lambda\leq\mathbf{d}x.\\
\label{cadef}x_\lambda\carrow x\qquad\Leftrightarrow\qquad\mathbf{d}^\mathrm{op}(x_\lambda)\geq\mathbf{d}^\mathrm{op}x\qquad\Leftrightarrow\qquad&\liminf_\lambda x_\lambda\mathbf{d}\geq x\mathbf{d}.
\end{align}
As $x_\lambda\barrowc x$ if and only if $x_\lambda\barrow x$ and $x_\lambda\arrowc x$, by \eqref{joinconverge}, and $r_\lambda\rightarrow r$ if and only if $\limsup_\lambda r_\lambda\leq r\leq\liminf_\lambda r_\lambda$, and likewise for $x_\lambda\carrowb x$, we have
\begin{align}
\label{baclim}x_\lambda\barrowc x\qquad&\Leftrightarrow\qquad\lim_\lambda\mathbf{d}x_\lambda=\mathbf{d}x.\\
\label{cablim}x_\lambda\carrowb x\qquad&\Leftrightarrow\qquad\lim_\lambda x_\lambda\mathbf{d}=x\mathbf{d}.
\end{align}

In general, these convergence notions depend on all $\mathbf{d}$ values, not just the small ones.  In particular, without extra assumptions, they can not be characterized by statements like $x_\lambda\mathbf{d}x\rightarrow0$ familiar from metric space theory.  However, there are still some general relationships of this sort.

\begin{prp}\label{convchar}
\begin{align}
\label{ac}x_\lambda\underline{\mathbf{d}}x\rightarrow0\quad&\Rightarrow\quad x_\lambda\arrowc x.\\
\label{ab}x_\lambda\overline{\mathbf{d}}x\rightarrow0\quad&\Rightarrow\quad x_\lambda\arrowb x.\\
\label{ab<=}x_\lambda\mathbf{d}x\rightarrow0\quad&\Leftarrow\quad x_\lambda\arrowb x\leq^\mathbf{d}x.\\
\label{abvee}x_\lambda\underline{\mathbf{d}}^\vee x\rightarrow0\quad&\Rightarrow\quad x_\lambda\barrowc x.
\end{align}
\end{prp}

\begin{proof}  Recall from \autoref{hemiprop} that $\mathbf{d}=\overline{\mathbf{d}}\circ\mathbf{d}=\mathbf{d}\circ\underline{\mathbf{d}}$.
\begin{itemize}
\item[\eqref{ac}]  If $x_\lambda\mathbf{\underline{d}}x\rightarrow0$ then $c\mathbf{d}x\leq\liminf_\lambda(c\mathbf{d}x_\lambda+x_\lambda\mathbf{\underline{d}}x)=c\mathbf{d}(x_\lambda)$.

\item[\eqref{ab}]  If $x_\lambda\mathbf{\overline{d}}x\rightarrow0$ then, as $\mathbf{d}=\overline{\mathbf{d}}\circ\mathbf{d}$ yields $x_\lambda\mathbf{\overline{d}}x+x\mathbf{d}c\geq x_\lambda\mathbf{d}c$, \eqref{a<b+c} yields $x\mathbf{d}c\geq\limsup_\lambda(x_\lambda\mathbf{d}c-x_\lambda\mathbf{\overline{d}}x)=(x_\lambda)\mathbf{d}c$.

\item[\eqref{ab<=}]  If $x_\lambda\arrowb x\leq^\mathbf{d}x$ then $\limsup_\lambda x_\lambda\mathbf{d}x=(x_\lambda)\mathbf{d}x\leq x\mathbf{d}x=0$.

\item[\eqref{abvee}]  If $x_\lambda\underline{\mathbf{d}}^\vee x\rightarrow0$ then $x_\lambda\underline{\mathbf{d}}x\rightarrow0$ so $x_\lambda\arrowc x$, by \eqref{ac}, but also $x_\lambda\overline{\mathbf{d}^\mathrm{op}}x=x_\lambda\underline{\mathbf{d}}^\mathrm{op}x\rightarrow0$ so $x_\lambda\barrow x$, by \eqref{ab}.
\qedhere
\end{itemize}
\end{proof}

In \cite{Goubault2013} Definition 7.1.15, any $x$ with $(x_\lambda)\mathbf{d}=x\mathbf{d}$ is called a \emph{$\mathbf{d}$-limit} of $(x_\lambda)$ (these are called \emph{forward limits} in \cite{Bonsangue1998} before Proposition 3.3 and just \emph{limits} in \cite{KunziSchellekens2002} Definition 11).  In general, $\mathbf{d}$-limits are not true limits in any topological sense, as they are not preserved by taking subnets.  For example, if we consider $x\mathbf{d}y=(x-y)_+$ on $\{0,1\}$ and take the sequence $(x_n)$ defined by $x_{2n}=0$ and $x_{2n+1}=1$, for all $n$, then $(x_n)\mathbf{d}=1\mathbf{d}$ while $(x_{2n})\mathbf{d}=0\mathbf{d}$.  But for $\mathbf{\overline{d}}$-pre-Cauchy nets, $\mathbf{d}$-limits are $\mathbf{d}^\circ_\bullet$-limits.

\begin{prp}\label{dlimits}
If $(x_\lambda)$ is $\mathbf{\overline{d}}$-pre-Cauchy with subnet $(y_\gamma)$ then
\begin{align}
\label{cab}x_\lambda\carrowb x\quad&\Leftrightarrow\quad(x_\lambda)\mathbf{d}=x\mathbf{d}\quad\Leftrightarrow\quad y_\gamma\carrowb x.\\
\intertext{If $(x_\lambda)$ is $\mathbf{\underline{d}}$-pre-Cauchy with subnet $(y_\gamma)$ then}
\label{bac}x_\lambda\barrowc x\quad&\Leftrightarrow\quad\mathbf{d}(x_\lambda)=\mathbf{d}x\quad\Leftrightarrow\quad y_\gamma\barrowc x.\\
\label{bac2}x_\lambda\barrowc x\quad&\Rightarrow\quad(x_\lambda)\underline{\mathbf{d}}=x\underline{\mathbf{d}}.\\
\label{barrowc}x_\lambda\barrowc x\quad&\Leftarrow\quad (x_\lambda)\underline{\mathbf{d}}=x\underline{\mathbf{d}}\quad\text{and}\quad x_\lambda\barrowc y,\text{ for some }y\in X.\\
\intertext{If $(x_\lambda)$ is $\mathbf{d}$-pre-Cauchy and $\mathbf{d}$ is a distance then}
\label{arrowc}x_\lambda\arrowc x\quad&\Leftrightarrow\quad x_\lambda\mathbf{d}x\rightarrow0.\\
\label{carrowc}x_\lambda\carrowc x\quad&\Leftrightarrow\quad x_\lambda\carrowb x\leq^\mathbf{d}x.
\end{align}
\end{prp}

\begin{proof}\
\begin{itemize}
\item[\eqref{cab}]  If $x_\lambda\carrowb x$, i.e. $\lim_\lambda x_\lambda\mathbf{d}=x\mathbf{d}$ (see \eqref{cablim}) then certainly $\limsup_\lambda x_\lambda\mathbf{d}=x\mathbf{d}$, i.e. $(x_\lambda)\mathbf{d}=x\mathbf{d}$.  Conversely, if $\limsup_\lambda x_\lambda\mathbf{d}=x\mathbf{d}$ then $\lim_\lambda x_\lambda\mathbf{d}=x\mathbf{d}$, as $x_\lambda\mathbf{d}$ converges, by \autoref{Clim} \eqref{dunder}.  Likewise, as $x_\lambda\mathbf{d}$ converges, $\lim_\lambda x_\lambda\mathbf{d}=\lim y_\gamma\mathbf{d}$, for any subnet $(y_\gamma)$, so $\lim_\lambda x_\lambda\mathbf{d}=x\mathbf{d}$ if and only if $\lim_\gamma y_\gamma\mathbf{d}=x\mathbf{d}$.

\item[\eqref{bac}]  Apply \autoref{Clim} \eqref{dover} as above.

\item[\eqref{bac2}]  By \eqref{Climeq} and \eqref{bac},
\[(x_\lambda)\mathbf{\underline{d}}y=\sup_{z\in X}(z\mathbf{d}y-z\mathbf{d}(x_\lambda))_+=\sup_{z\in X}(z\mathbf{d}y-z\mathbf{d}x)_+=x\underline{\mathbf{d}}y.\]

\item[\eqref{barrowc}]  As $(x_\lambda)\underline{\mathbf{d}}x=x\underline{\mathbf{d}}x=0$, \eqref{ac} yields $x_\lambda\arrowc x$.  On the other hand, $x\underline{\mathbf{d}}y=(x_\lambda)\underline{\mathbf{d}}y=y\underline{\mathbf{d}}y=0$, where the second equality follows \eqref{bac2} and the $x_\lambda\barrowc y$ assumption.  As $\mathbf{d}\leq\mathbf{d}\circ\underline{\mathbf{d}}$, it follows that $\mathbf{d}y\leq\mathbf{d}x+x\underline{\mathbf{d}}y=\mathbf{d}x$.  Then $(x_\lambda)\mathbf{d}^\mathrm{op}=\mathbf{d}(x_\lambda)=\mathbf{d}y\leq\mathbf{d}x=x\mathbf{d}^\mathrm{op}$, i.e. $x_\lambda\barrow x$, where the first and second equalities follow from \eqref{baclim}.

\item[\eqref{arrowc}]  By \eqref{epreC}, $(x_\lambda)\mathbf{d}=(x_\lambda)\underline{\mathbf{d}}$.  Thus it suffices to prove
\[x_\lambda\arrowc x\quad\Leftrightarrow\quad x_\lambda\underline{\mathbf{d}}x\rightarrow0,\]
for $\underline{\mathbf{d}}$-pre-Cauchy $(x_\lambda)$.  The $\Leftarrow$ part is \eqref{ac}.  Conversely, \eqref{Climeq} yields
\[(x_\lambda)\mathbf{\underline{d}}x=\sup_{z\in X}(z\mathbf{d}x-z\mathbf{d}(x_\lambda))_+\leq\sup_{z\in X}(z\mathbf{d}x-z\mathbf{d}x)_+=0,\]
where the inequality follows from $x_\lambda\arrowc x$ and \eqref{acdef}.

\item[\eqref{carrowc}]  If $x_\lambda\carrowc x$ then $x_\lambda\mathbf{d}x\rightarrow0$, by \eqref{arrowc}, so $x\mathbf{d}x\leq(x_\lambda)\mathbf{d}x=0$, i.e. $x\leq^\mathbf{d}x$.  Also $x_\lambda\overline{\mathbf{d}}x\leq x_\lambda\mathbf{d}x\rightarrow0$, as $\overline{\mathbf{d}}\leq\mathbf{d}$ by \eqref{dis}, so $x_\lambda\arrowb x$, by \eqref{ab}.  This proves $\Rightarrow$, while \eqref{ab<=} and \eqref{arrowc} prove $\Leftarrow$.\qedhere
\end{itemize}
\end{proof}

For hemimetric $\mathbf{d}$, \eqref{ab} and \eqref{ab<=} show that $\mathbf{d}_\bullet$-convergence is equivalent to the statement $x_\lambda\mathbf{d}x\rightarrow0$ familiar from metric space theory.  But for general distance $\mathbf{d}$, it is rather $\mathbf{d}_\circ$-convergence that is characterized by $x_\lambda\mathbf{d}x\rightarrow0$, at least for $\mathbf{d}$-pre-Cauchy nets, by \eqref{arrowc}.

Also note \eqref{bac2} and \eqref{barrowc} describe a close relationship between $\mathbf{d}^\bullet_\circ$-limits and $\underline{\mathbf{d}}^\circ_\circ$-limits of $\underline{\mathbf{d}}$-pre-Cauchy $(x_\lambda)$ (as \eqref{cab} and \eqref{carrowc} show $x=\underline{\mathbf{d}}^\circ_\circ$-$\lim x_\lambda$ iff $(x_\lambda)\underline{\mathbf{d}}=x\underline{\mathbf{d}}$).  Namely, every $\mathbf{d}^\bullet_\circ$-limit of a $\underline{\mathbf{d}}$-pre-Cauchy net $(x_\lambda)$ is a $\underline{\mathbf{d}}^\circ_\circ$-limit, by \eqref{bac2}, while conversely the mere existence of a $\mathbf{d}^\bullet_\circ$-limit guarantees that any $\underline{\mathbf{d}}^\circ_\circ$-limit is a $\mathbf{d}^\bullet_\circ$-limit, by \eqref{barrowc}.

For a simple example of a $\mathbf{d}$-Cauchy net where $x_\lambda\carrowb x\nleq^\mathbf{d}x$ and hence $x_\lambda\not\hspace{-4pt}\carrowc x$, take any $x_\lambda\rightarrow0<x$ in $[0,\infty]$, taking $y\mathbf{d}z=z$ for $\mathbf{d}$.

\section{\bf Directed Subsets}\label{DirectedSubsets}

Directed subsets play a fundamental role in domain theory.  These correspond to increasing nets which are generalized by the (pre-)Cauchy-nets above, and this is usually considered the only path to quantitative domain theory.  However, an equally valid but subtly different theory can be obtained from a more direct generalization of directed subsets.

\begin{dfn}\label{ddirdef}
We call $Y\subseteq X$ \emph{$\mathbf{d}$-directed} if, for all finite $F\subseteq Y$,
\[\inf_{y\in Y}\sup_{x\in F}x\mathbf{d}y=0.\]
\end{dfn}
Equivalently, $Y$ is $\mathbf{d}$-directed if and only if
\[\forall\epsilon>0\ \forall F\in\mathcal{F}(Y)\ \exists y\in Y\ \forall x\in F\ (x<^\mathbf{d}_\epsilon y),\]
where $\mathcal{F}(Y)$ again denotes the finite subsets of $Y$.  In particular, for any transitive relation $\sqsubset$, $Y$ is $\sqsubset$-directed iff every finite subset of $Y$ has an upper bound w.r.t. $\sqsubset$, i.e. iff $Y$ is directed in the usual sense.

It will also be convenient to consider the following weaker notion obtained by restricting to singleton $F$.

\begin{dfn}\label{dfindef}
We call $Y\subseteq X$ \emph{$\mathbf{d}$-final} if, for all $x\in Y$,
\[\inf_{y\in Y}x\mathbf{d}y=0.\]
\end{dfn}
Equivalently, $Y$ is $\mathbf{d}$-final if and only if
\[\forall\epsilon>0\ \forall x\in Y\ \exists y\in Y\ (x<^\mathbf{d}_\epsilon y).\]
In particular, for any transitive relation $\sqsubset$, $Y$ is $\sqsubset$-final iff every single element $x$ has an upper bound $y\sqsupset x$.  In \cite{Keimel2016}, $\sqsubset$-final subsets are called `cofinal', while in \cite[Proposition III-4.3]{GierzHofmannKeimelLawsonMisloveScott2003} and \cite[Proposition 5.13]{Goubault2013} they are called `rounded', at least in the ideal case.  Note arbitrary subsets are $\mathbf{d}$-final when $\leq^\mathbf{d}$ is reflexive.  In particular, arbitrary subsets are $\sqsubseteq$-final when $\sqsubseteq$ is a preorder.

As with nets, let us adopt the convention that sets written on the left/right of a function denote suprema/infima, so
\begin{align*}
Z\mathbf{d}x&=\sup_{z\in Z}z\mathbf{d}x.\\
x\mathbf{d}Z&=\inf_{z\in Z}x\mathbf{d}z.
\end{align*}
Again we extend this to unary functions, i.e.
\begin{align*}
Z\mathbf{d}&=\sup_{z\in Z}z\mathbf{d}.\\
\mathbf{d}Z&=\inf_{z\in Z}\mathbf{d}z.
\end{align*}
For example, applying these conventions twice, for any $Y,Z\subseteq X$ we have
\begin{align*}
(Y\mathbf{d})Z&=(\sup_{y\in Y}y\mathbf{d})Z=\inf_{z\in Z}\sup_{y\in Y}y\mathbf{d}z.\\
Y(\mathbf{d}Z)&=Y(\inf_{z\in Z}\mathbf{d}z)=\sup_{y\in Y}\inf_{z\in Z}y\mathbf{d}z.
\end{align*}
So the definition of $\mathbf{d}$-directedness can thus be restated as follows
\[Y\text{ is \emph{$\mathbf{d}$-directed}}\qquad\Leftrightarrow\qquad\forall F\in\mathcal{F}(Y)\ (F\mathbf{d})Y=0.\]
In fact, for $\mathbf{d}$-directed $Y$, it does not matter where we put the parentheses.

\begin{prp}
If $\mathbf{d}$ is a distance and $Y$ is $\mathbf{d}$-final then
\begin{align}
\label{FdY}\forall F\in\mathcal{F}(X)\ (F\mathbf{d})Y=F(\mathbf{d}Y)\qquad&\Leftrightarrow\qquad Y\text{ is $\mathbf{d}$-directed}.\\
\label{YdYd}\overline{\mathbf{d}}Y=\mathbf{d}Y\qquad&\text{and}\qquad Y\underline{\mathbf{d}}=Y\mathbf{d}.
\end{align}
\end{prp}

\begin{proof}\
\begin{itemize}
\item[\eqref{FdY}]  If $Y$ is $\mathbf{d}$-final and, for all $F\in\mathcal{F}(X)$, $(F\mathbf{d})Y=F(\mathbf{d}Y)$ then in particular, for all $F\in\mathcal{F}(Y)$, we have $(F\mathbf{d})Y=F(\mathbf{d}Y)=0$, i.e. $Y$ is $\mathbf{d}$-directed.

For each $x\in F$, $x\mathbf{d}Y\leq(F\mathbf{d})Y$ so $F(\mathbf{d}Y)\leq(F\mathbf{d})Y$.  Conversely, say $Y$ is $\mathbf{d}$-directed and take $\epsilon>0$.  For each $x\in F$, we have $x'\in Y$ with $x\mathbf{d}x'\leq x\mathbf{d}Y+\epsilon\leq F(\mathbf{d}Y)+\epsilon$.  Then we can take $y\in Y$ with $F'\mathbf{d}y<\epsilon$, where $F'=\{x':x\in F\}$.  If $\mathbf{d}$ is a distance then $F\mathbf{d}y\leq F(\mathbf{d}Y)+2\epsilon$.  As $\epsilon>0$ was arbitrary, $(F\mathbf{d})Y\leq F(\mathbf{d}Y)$.

\item[\eqref{YdYd}]  If $\mathbf{d}$ is a distance then $Y\underline{\mathbf{d}}\leq Y\mathbf{d}$, by \eqref{dis}.  Conversely, note first that $\inf\limits_{r\in R,s\in S}(r+s)\leq\inf R+\sup S$, for all $R,S\subseteq[0,\infty]$, so
\[y(\mathbf{d}\circ\mathbf{\underline{d}})z\leq\inf_{w\in Y}(y\mathbf{d}w+w\underline{\mathbf{d}}z)\leq y\mathbf{d}Y+Y\underline{\mathbf{d}}z.\]
So if $Y$ is also $\mathbf{d}$-final then
\[Y\mathbf{d}z\leq Y(\mathbf{d}\circ\mathbf{\underline{d}})z=\sup_{y\in Y}y(\mathbf{d}\circ\mathbf{\underline{d}})z\leq\sup_{y\in Y}(y\mathbf{d}Y+Y\underline{\mathbf{d}}z)=Y\underline{\mathbf{d}}z.\]
Likewise $\overline{\mathbf{d}}Y\leq\mathbf{d}Y$, by \eqref{dis}, and conversely
\[z\mathbf{d}Y\leq z(\mathbf{\overline{d}}\circ\mathbf{d})Y\leq\inf_{x,y\in Y}(z\mathbf{\overline{d}}x+x\mathbf{d}y)=\inf_{x\in Y}(z\overline{\mathbf{d}}x+x\mathbf{d}Y)=z\overline{\mathbf{d}}Y.\qedhere\]
\end{itemize}
\end{proof}

Recall the standard topological notion of separability, namely that $X$ is \emph{$\mathcal{T}$-separable}, for some topology $\mathcal{T}$ on $X$, if $X$ contains a countable $\mathcal{T}$-dense subset $Y$, i.e. if every non-empty $O\in\mathcal{T}$ contains some $y\in Y$.

\begin{prp}\label{finsep}
If $\mathbf{d}$ is a distance then
\[X\text{ is \emph{$\mathbf{d}$-final and $\mathbf{d}^\bullet$-separable}}\quad\Leftrightarrow\quad X(\mathbf{d}Z)=0\text{ for some countable }Z.\]
\end{prp}

\begin{proof}
Assume $Z$ is $\mathbf{d}^\bullet$-dense in $X$.  If $X$ is $\mathbf{d}$-final then, for all $x\in X$ and $\epsilon>0$, $x^\bullet_\epsilon$ is non-empty and hence contains some $z\in Z$, i.e. $X(\mathbf{d}Z)=0$.  If $X$ is $\mathbf{d}^\bullet$-separable then we can choose $Z$ to be countable, proving $\Rightarrow$.

Conversely, if $X(\mathbf{d}Z)=0$ then certainly $X(\mathbf{d}X)=0$, i.e. $X$ is $\mathbf{d}$-final.  And if $O=(x_1)^\bullet_{\epsilon_1}\cap\cdots\cap(x_n)^\bullet_{\epsilon_n}$ is non-empty, for some $x_1,\cdots,x_n\in X$ and $\epsilon_1,\cdots,\epsilon_n>0$, then we can take $x\in O$ and $\epsilon>0$ such that $x_k\mathbf{d}x+\epsilon<\epsilon_k$, for all $k\leq n$.  As $\mathbf{d}$ is a distance, this means $x^\bullet_\epsilon\subseteq O$.  As $X(\mathbf{d}Z)=0$, we have some $z\in Z$ with $z\in x^\bullet_\epsilon\subseteq O$, so $Z$ is indeed dense in $X$.
\end{proof}

It will be useful to define what it means for a subset to be below a net and vice versa.  Specifically, for any $(x_\lambda)\subseteq X$ and $Y\subseteq X$, let
\begin{align*}
(x_\lambda)\leq^\mathbf{d}Y\qquad&\Leftrightarrow\qquad x_\lambda\mathbf{d}Y\rightarrow0.\\
Y\leq^\mathbf{d}(x_\lambda)\qquad&\Leftrightarrow\qquad\,y\mathbf{d}x_\lambda\rightarrow0,\text{ for all }y\in Y.
\end{align*}

\begin{prp}\label{Yxlam}
For any $(x_\lambda)\subseteq X$ and $Y\subseteq X$,
\begin{align}
\label{dYdxlam} Y\leq^\mathbf{d}(x_\lambda)\qquad&\Rightarrow\qquad Y\mathbf{d}\leq(x_\lambda)\underline{\mathbf{d}}\quad\text{and}\quad\overline{\mathbf{d}}Y\geq\mathbf{d}(x_\lambda).\\
\label{Ydxlamd} Y\geq^\mathbf{d}(x_\lambda)\qquad&\Rightarrow\qquad Y\underline{\mathbf{d}}\geq(x_\lambda)\mathbf{d}\quad\text{and}\quad\mathbf{d}Y\leq\overline{\mathbf{d}}(x_\lambda).
\end{align}
\end{prp}

\begin{proof}\
\begin{itemize}
\item[\eqref{dYdxlam}]  As $\mathbf{d}=\mathbf{d}\circ\underline{\mathbf{d}}=\overline{\mathbf{d}}\circ\mathbf{d}$, by \autoref{hemiprop}, $Y\leq^\mathbf{d}(x_\lambda)$ yields
\begin{gather*}
Y\mathbf{d}=\sup_{y\in Y}y\mathbf{d}\leq\sup_{y\in Y}\liminf_\lambda(y\mathbf{d}x_\lambda+x_\lambda\underline{\mathbf{d}})\leq(x_\lambda)\underline{\mathbf{d}}.\\
\mathbf{d}(x_\lambda)=\liminf_\lambda\mathbf{d}x_\lambda\leq\inf_{y\in Y}\liminf_\lambda(\overline{\mathbf{d}}y+y\mathbf{d}x_\lambda)=\overline{\mathbf{d}}Y.
\end{gather*}
\item[\eqref{Ydxlamd}]  Again as $\mathbf{d}=\mathbf{d}\circ\underline{\mathbf{d}}=\overline{\mathbf{d}}\circ\mathbf{d}$, by \autoref{hemiprop}, $Y\geq^\mathbf{d}(x_\lambda)$ yields
\begin{gather*}
(x_\lambda)\mathbf{d}=\limsup_\lambda x_\lambda\mathbf{d}\leq\limsup_\lambda(x_\lambda\mathbf{d}Y+Y\underline{\mathbf{d}})=Y\underline{\mathbf{d}}.\\
\mathbf{d}Y=\inf_{y\in Y}\mathbf{d}Y\leq\liminf_\lambda\inf_{y\in Y}(\overline{\mathbf{d}}x_\lambda+x_\lambda\mathbf{d}y)=\overline{\mathbf{d}}(x_\lambda).\qedhere
\end{gather*}
\end{itemize}
\end{proof}
Note that if $(y_\gamma)$ is a subnet of $(x_\lambda)$ then
\begin{align*}
(x_\lambda)\leq^\mathbf{d}Y\qquad&\Rightarrow\qquad(y_\gamma)\leq^\mathbf{d}Y.\\
Y\leq^\mathbf{d}(x_\lambda)\qquad&\Rightarrow\qquad Y\leq^\mathbf{d}(y_\gamma).
\end{align*}
The converses also hold for pre-Cauchy nets.

\begin{prp}\label{subnet<Y}
If $(y_\gamma)$ is a subnet of $(x_\lambda)$ then
\begin{align}
\label{ygamma<Y}(x_\lambda)\text{ is $\overline{\mathbf{d}}$-pre-Cauchy and }(y_\gamma)\leq^\mathbf{d}Y\qquad&\Rightarrow\qquad(x_\lambda)\leq^\mathbf{d}Y.\\
\label{Y<ygamma}(x_\lambda)\text{ is $\underline{\mathbf{d}}$-pre-Cauchy and }Y\leq^\mathbf{d}(y_\gamma)\qquad&\Rightarrow\qquad Y\leq^\mathbf{d}(x_\lambda).
\end{align}
\end{prp}

\begin{proof}\
\begin{itemize}
\item[\eqref{ygamma<Y}]  Assume $(x_\lambda)$ is $\overline{\mathbf{d}}$-pre-Cauchy and $(y_\gamma)\leq^\mathbf{d}Y$.  Then
\begin{align*}
x_\lambda\mathbf{d}Y&\leq\limsup_\gamma(x_\lambda\overline{\mathbf{d}}y_\gamma+y_\gamma\mathbf{d}Y)\\
&=\limsup_\gamma(x_\lambda\overline{\mathbf{d}}y_\gamma)\qquad\text{as }(y_\gamma)\leq^\mathbf{d}Y\\
&\leq\limsup_\delta(x_\lambda\overline{\mathbf{d}}x_\delta)\qquad\text{as $(y_\gamma)$ is a subnet}\\
&\rightarrow0\qquad\text{as $(x_\lambda)$ is $\overline{\mathbf{d}}$-pre-Cauchy}.
\end{align*}
Thus $(x_\lambda)\leq^\mathbf{d}Y$.

\item[\eqref{Y<ygamma}]  If $(x_\lambda)$ is $\underline{\mathbf{d}}$-pre-Cauchy then $y\mathbf{d}x_\lambda$ has a limit, for any $y$, by \autoref{Clim} \eqref{dover}.  So if $y\mathbf{d}y_\gamma=0$, for some subnet $(y_\gamma)$, this limit must be $0$.  Applied to all $y\in Y$, we see that $Y\leq^\mathbf{d}(y_\gamma)$ implies $Y\leq^\mathbf{d}(x_\lambda)$.\qedhere
\end{itemize}
\end{proof}

Defining $Y\leq^\mathbf{d}x$ to mean $y\leq^\mathbf{d}x$, for all $y\in Y$, we also see that
\begin{equation}\label{x<Yd}
Y\leq^\mathbf{d}(x_\lambda)\quad\text{and}\quad x_\lambda\arrowc x\qquad\Rightarrow\qquad Y\leq^\mathbf{d}x.
\end{equation}
Indeed if $y\in Y\leq^\mathbf{d}(x_\lambda)$ and $x_\lambda\arrowc x$ then $y\mathbf{d}x\leq y\mathbf{d}(x_\lambda)=0$, by \eqref{acdef}.

Different versions of quantitative domain theoretic concepts are connected via results about $\mathbf{d}$-directed subsets having equivalent $\mathbf{d}$-pre-Cauchy nets (and vice versa, a topic we will return to in \autoref{Completeness}).  Specifically, let
\[Y\equiv^\mathbf{d}(x_\lambda)\qquad\Leftrightarrow\qquad Y\leq^\mathbf{d}(x_\lambda)\leq^\mathbf{d}Y.\]

\begin{prp}\label{directedCauchy}
For any $Y\subseteq X$,
\begin{align}
\label{Yddir}\exists\text{ $\mathbf{d}$-Cauchy }(x_\lambda)\equiv^\mathbf{d}Y\qquad&\Leftarrow\qquad Y\text{ is $\mathbf{d}$-directed}.\\
\intertext{If $\mathbf{d}$ is a distance then}
\label{x<YdpC}(x_\lambda)\equiv^\mathbf{d}Y\qquad&\Rightarrow\qquad(x_\lambda)\text{ is $\mathbf{d}$-pre-Cauchy}.\\
\label{x<Yddir}\exists(x_\lambda)\equiv^\mathbf{d}Y\qquad&\Leftrightarrow\qquad Y\text{ is $\mathbf{d}$-directed}.
\intertext{If $\mathbf{d}$ is a distance and $X$ is $\overline{\mathbf{d}}^\bullet_\bullet$-separable then}
\label{sepdir}\exists(x_n)_{n\in\mathbb{N}}\equiv^\mathbf{d}Y\qquad&\Leftrightarrow\qquad Y\text{ is $\mathbf{d}$-directed}.
\end{align}
\end{prp}

\begin{proof}\
\begin{itemize}
\item[\eqref{Yddir}] If $Y$ is $\mathbf{d}$-directed then, for $F\in\mathcal{F}(Y)$ and $\epsilon>0$, take $y_{F,\epsilon}\in Y$ with $F\mathbf{d}y_{F,\epsilon}<\epsilon$.  Ordering $\mathcal{F}(Y)\times(0,\infty)$ by $\subseteq\times\geq$, we get $(y_{F,\epsilon})\subseteq Y\leq^\mathbf{d}(y_{F,\epsilon})$.  In particular, $(y_{F,\epsilon})$ is $\mathbf{d}$-pre-Cauchy.  By \autoref{Clim} \eqref{preCauchysub}, we can replace $(y_{F,\epsilon})$ with a $\mathbf{d}$-Cauchy subnet.  Lastly, note $(y_{F,\epsilon})\subseteq Y$ implies $(y_{F,\epsilon})(\mathbf{d}Y)\leq Y(\mathbf{d}Y)=0$, as $Y$ is $\mathbf{d}$-directed and hence $\mathbf{d}$-final, i.e. $(y_{F,\epsilon})\leq^\mathbf{d}Y$.

\item[\eqref{x<YdpC}]  If $(x_\lambda)\equiv^\mathbf{d}Y$ then, as $\mathbf{d}$ is a distance,
\begin{align*}
\limsup_\gamma\limsup_\delta x_\gamma\mathbf{d}x_\delta&\leq\limsup_\gamma\inf_{y\in Y}\limsup_\delta(x_\gamma\mathbf{d}y+y\mathbf{d}x_\delta)\\
&=\limsup_\gamma\inf_{y\in Y}(x_\gamma\mathbf{d}y+\limsup_\delta y\mathbf{d}x_\delta)\\
&=\limsup_\gamma x_\gamma\mathbf{d}Y\quad\text{as }Y\leq^\mathbf{d}(x_\lambda)\\
&=0\quad\text{as }(x_\lambda)\leq^\mathbf{d}Y.
\end{align*}

\item[\eqref{x<Yddir}]  If $(x_\lambda)\equiv^\mathbf{d}Y$ then, for any $F\in\mathcal{F}(Y)$,
\begin{align*}
(F\mathbf{d})Y&=\inf_{y\in Y}\sup_{x\in F}x\mathbf{d}y\\
&\leq\liminf_\lambda\inf_{y\in Y}\sup_{x\in F}(x\mathbf{d}x_\lambda+x_\lambda\mathbf{d}y)\\
&=\liminf_\lambda(F\mathbf{d}x_\lambda+x_\lambda\mathbf{d}Y)\\
&\leq(F\mathbf{d})(x_\lambda)+(x_\lambda)(\mathbf{d}Y)\\
&=0,
\end{align*}
as $(F\mathbf{d})(x_\lambda)=0$ because $Y\leq^\mathbf{d}(x_\lambda)$ and $(x_\lambda)(\mathbf{d}Y)=0$ because $(x_\lambda)\leq^\mathbf{d}Y$.  This shows $Y$ is $\mathbf{d}$-directed.  The converse is \eqref{Yddir}.

\item[\eqref{sepdir}]  Assume $\mathbf{d}$ is a distance, $X$ is $\overline{\mathbf{d}}^\bullet_\bullet$-separable and $Y$ is $\mathbf{d}$-directed.  As $\overline{\mathbf{d}}$ is hemimetric, $\overline{\mathbf{d}}^\bullet_\bullet=\overline{\mathbf{d}}^{\vee\bullet}$, by \cite[Proposition 6.1.19]{Goubault2013}.  Also $X$ is trivially $\overline{\mathbf{d}}^\vee$-final, so we have countable $Z\subseteq X$ with $X(\overline{\mathbf{d}}^\vee Z)=0$, by \autoref{finsep}.  Let $(z_n)_{n\in\mathbb{N}}$ enumerate $Z$ (note we do not consider $0$ to be an element of $\mathbb{N}$).  For each $n\in\mathbb{N}$, we can take $y_1,\cdots,y_n\in Y$ with $z_k\mathbf{d}y_k<z_k\mathbf{d}Y+1/n$, for all $k\leq n$.  Applying \autoref{ddirdef} to $F=\{y_1,\cdots,y_n\}$, we obtain $x_n\in Y$ with $F\mathbf{d}x_n<1/n$.  As $\mathbf{d}$ is a distance, this implies that $z_k\mathbf{d}x_n<z_k\mathbf{d}Y+2/n$, for all $k\leq n$.  For any $y\in Y$ and $\epsilon>0$, we have $N\in\mathbb{N}$ with $y(\overline{\mathbf{d}}^\vee)z_N<\epsilon$ and hence $z_N\mathbf{d}Y\leq z_N\overline{\mathbf{d}}y+y\mathbf{d}Y<\epsilon$, as $Y$ is $\mathbf{d}$-final.  Thus, for any $n\geq N$,
\[y\mathbf{d}x_n\leq y\overline{\mathbf{d}}z_N+z_N\mathbf{d}x_n<\epsilon+z_N\mathbf{d}Y+2/n\leq 2\epsilon+2/n.\]
As $\epsilon>0$ was arbitrary, $y\mathbf{d}x_n\rightarrow0$, so $(x_n)\subseteq Y\leq^\mathbf{d}(x_n)$.  This completes the proof of $\Leftarrow$, while $\Rightarrow$ follows from \eqref{x<Yddir}.\qedhere
\end{itemize}
\end{proof}

Mostly we use $\mathbf{d}$-directed subsets, but they can be replaced by $\mathbf{d}$-ideals.

\begin{dfn}\label{idealdef}
We call $I\subseteq X$ a \emph{$\mathbf{d}$-ideal} if, for all $F\in\mathcal{F}(X)$,
\[F\subseteq I\quad\Leftrightarrow\quad(F\mathbf{d})I=0.\]
\end{dfn}

Note that for the $\Leftarrow$ part it suffices to consider singleton $F$, i.e.
\begin{equation}\label{singletonleft}
x\in I\quad\Leftarrow\quad x\mathbf{d}I=0.
\end{equation}
For if $(F\mathbf{d})I=0$ then certainly $x\mathbf{d}I=0$, for all $x\in F$, so \eqref{singletonleft} yields $x\in I$, for all $x\in F$, and hence $F\subseteq I$.

\begin{prp}\label{dballclosureprp}
For distance $\mathbf{d}$, the $\overline{\mathbf{d}}^\bullet$-closure of $\mathbf{d}$-final $Y\subseteq X$ is
\begin{equation}\label{dballclosure}
\overline{Y}^{\overline{\bullet}}=\{x\in X:x\mathbf{d}Y=0\}.
\end{equation}
If $Y$ is $\mathbf{d}$-directed then $\overline{Y}^{\overline{\bullet}}$ is the smallest $\mathbf{d}$-ideal containing $Y$.
\end{prp}

\begin{proof}
Assume $\mathbf{d}$ is a distance and $x\mathbf{d}Y=0$.  Then whenever we have $c_1,\cdots,c_n\in X$ and $r_1,\cdots,r_n\in(0,\infty)$ with $x\in(c_1)_{r_1}^\bullet\cap\ldots\cap(c_n)_{r_n}^\bullet$, we can always find $y\in Y$ with $x\mathbf{d}y<(r_1-c_1\mathbf{d}x)\wedge\ldots\wedge(r_1-c_1\mathbf{d}x)$, as $x\mathbf{d}Y=0$.  It follows that $y\in(c_1)_{r_1}^\bullet\cap\ldots\cap(c_n)_{r_n}^\bullet$, as $\mathbf{d}$ is a distance.  Thus $x\in\overline{Y}^\bullet(=$ the $\mathbf{d}^\bullet$-closure of $Y)$.  Conversely, if $\leq^\mathbf{d}$ is reflexive and $x\mathbf{d}Y>\epsilon>0$ then $x^\bullet_\epsilon\cap Y=\emptyset$ while $x\in x^\bullet_\epsilon$, i.e. $x\notin\overline{Y}^\bullet$.  Thus if $\mathbf{d}$ is a hemimetric,
\[\overline{Y}^\bullet=\{x\in X:x\mathbf{d}Y=0\}.\]
If $\mathbf{d}$ is a distance and $Y$ is $\mathbf{d}$-final then \eqref{YdYd} and the above argument applied to the hemimetric $\overline{\mathbf{d}}$ shows the $\overline{\mathbf{d}}$-closure $\overline{Y}^{\overline{\bullet}}$ is given by \eqref{dballclosure}:
\[\overline{Y}^{\overline{\bullet}}=\{x\in X:x\overline{\mathbf{d}}Y=0\}=\{x\in X:x\mathbf{d}Y=0\}.\]
It follows that any $\mathbf{d}$-ideal $I$ containing $Y$ contains $\overline{Y}^{\overline{\bullet}}$, for if $0=x\mathbf{d}Y\geq x\mathbf{d}I$ then $x\in I$, by \eqref{singletonleft}.  But if $Y$ is $\mathbf{d}$-directed then, by \eqref{FdY},
\[F\subseteq\overline{Y}^{\overline{\bullet}}\quad\Leftrightarrow\quad F(\mathbf{d}Y)=(F\mathbf{d})Y=0\quad\Leftrightarrow\quad(F\mathbf{d})\overline{Y}^{\overline{\bullet}}=0.\]
For the last $\Leftrightarrow$, note that $(F\mathbf{d})\overline{Y}^{\overline{\bullet}}\leq(F\mathbf{d})Y$, as $Y\subseteq\overline{Y}^{\overline{\bullet}}$, and conversely
\[(F\mathbf{d})Y=\inf_{y\in Y}F\mathbf{d}y\leq\inf_{y\in Y,z\in\overline{Y}^{\overline{\bullet}}}(F\mathbf{d}z+z\mathbf{d}y)=\inf_{z\in\overline{Y}^{\overline{\bullet}}}(F\mathbf{d}z+z\mathbf{d}Y)=(F\mathbf{d})\overline{Y}^{\overline{\bullet}},\]
by \eqref{dballclosure}.  Thus $\overline{Y}^{\overline{\bullet}}$ itself is a $\mathbf{d}$-ideal.
\end{proof}

\begin{prp}\label{idealdiretedclosed}
If $I$ is $\mathbf{d}$-ideal then $I$ is $\mathbf{d}$-directed and $\overline{\mathbf{d}}^\bullet$-closed.  If $\mathbf{d}$ is a distance, any $\mathbf{d}$-directed $\overline{\mathbf{d}}^\bullet$-closed $I\subseteq X$ is a $\mathbf{d}$-ideal.
\end{prp}

\begin{proof}
If $I$ is $\mathbf{d}$-ideal then certainly $I$ is $\mathbf{d}$-directed.  In particular, $I$ is $\mathbf{d}$-final so $\mathbf{d}I\leq\overline{\mathbf{d}}I$ follows as in the proof of \eqref{YdYd}:
\[z\mathbf{d}I\leq z(\mathbf{\overline{d}}\circ\mathbf{d})I\leq\inf_{x,y\in I}(z\mathbf{\overline{d}}x+x\mathbf{d}y)=\inf_{x\in I}(z\overline{\mathbf{d}}x+x\mathbf{d}I)=z\overline{\mathbf{d}}I.\]
So if $x$ is in the $\overline{\mathbf{d}}^\bullet$-closure of $I$ then $x\mathbf{d}I\leq x\overline{\mathbf{d}}I=0$, by \eqref{dballclosure} (with $\overline{\mathbf{d}}$ replacing $\mathbf{d}$).  Thus $x\in I$, by the definition of $\mathbf{d}$-ideal, i.e. $I$ is $\overline{\mathbf{d}}^\bullet$-closed.

Conversely, assume $\mathbf{d}$ is a distance and $I$ is $\mathbf{d}$-directed and $\overline{\mathbf{d}}^\bullet$-closed.  In particular, the $\Rightarrow$ part of \autoref{idealdef} holds, as $I$ is $\mathbf{d}$-directed.  Also any $x$ with $x\mathbf{d}I=0$ is in $I$, by \eqref{dballclosure}, as $\mathbf{d}$ is a distance and $I$ is $\overline{\mathbf{d}}^\bullet$-closed and $\mathbf{d}$-final (even $\mathbf{d}$-directed).  This implies that the $\Leftarrow$ part of \autoref{idealdef} holds too, as noted in \eqref{singletonleft}.
\end{proof}

\section{\bf Upper Bounds}\label{UpperBounds}

Next we examine `$\mathbf{d}$-minimal' upper bounds of $\mathbf{d}$-directed subsets.

\begin{dfn}
Define \emph{$\mathbf{d}$-suprema} and \emph{$\mathbf{d}$-maxima} of $Y\subseteq X$ by
\begin{align}
\label{supdef}x=\text{$\mathbf{d}$-$\sup Y\ $}\qquad&\Leftrightarrow\qquad Y\leq^\mathbf{d}x\quad\text{and}\quad Y\mathbf{d}\geq x\mathbf{d}.\\
\label{maxdef}x=\text{$\mathbf{d}$-$\max Y$}\qquad&\Leftrightarrow\qquad Y\leq^\mathbf{d}x\quad\text{and}\quad\mathbf{d}Y\leq\mathbf{d}x.
\end{align}
\end{dfn}

Note $\mathbf{d}$-suprema and $\mathbf{d}$-maxima are not necessarily unique, so $=$ here is not really equality.  Put another way, we are officially taking $\mathbf{d}$-$\sup$ and $\mathbf{d}$-$\max$ as relations, not functions, and adding the $=$ symbol simply for consistency with standard supremum/maximum notation.  We consider $\mathbf{d}$-suprema and $\mathbf{d}$-maxima analogous to $\mathbf{d}^\circ_\circ$-limits and $\mathbf{d}^\bullet_\circ$-limits respectively, as indicated by the following analog of \autoref{dlimits}.

\begin{prp}\label{supmax=}
If $\mathbf{d}$ is a distance then, for any $Y\subseteq X$,
\begin{align}
\label{supxx}x=\text{$\mathbf{d}$-$\sup Y$}\qquad&\Leftrightarrow\qquad Y\mathbf{d}=x\mathbf{d}\quad\text{and}\quad x\leq^\mathbf{d}x.\\
\label{max=>sup}\text{$x=\mathbf{d}$-$\max Y$}\qquad&\Rightarrow\qquad\text{$x=\underline{\mathbf{d}}$-$\sup Y$}.\\
\intertext{If $\mathbf{d}$ is a distance and $Y\subseteq X$ is $\mathbf{d}$-final then}
\label{max1dir}\text{$x=\mathbf{d}$-$\max Y$}\qquad&\Leftrightarrow\qquad\mathbf{d}Y=\mathbf{d}x.\\
\label{Emax}\text{$x=\mathbf{d}$-$\max Y$}\qquad&\Leftarrow\qquad\text{$x=\underline{\mathbf{d}}$-$\sup Y\quad$and}\quad\text{$\exists\,y=\mathbf{d}$-$\max Y$}.
\end{align}
\end{prp}

\begin{proof}\
\begin{itemize}
\item[\eqref{supxx}]  If $Y\mathbf{d}=x\mathbf{d}$ and $x\leq^\mathbf{d}x$ then $Y\mathbf{d}x=x\mathbf{d}x=0$, i.e. $Y\leq^\mathbf{d}x$ so $x=\mathbf{d}$-$\sup Y$.  If $Y\leq^\mathbf{d}x$ and $x\mathbf{d}\leq Y\mathbf{d}$ then $x\mathbf{d}x\leq Y\mathbf{d}x=0$ and, as $\mathbf{d}$ is a distance, $Y\mathbf{d}\leq Y\mathbf{d}x+x\mathbf{d}=x\mathbf{d}$, i.e. $x\leq^\mathbf{d}x$ and $x\mathbf{d}=Y\mathbf{d}$.

\item[\eqref{max=>sup}]  If $\mathbf{d}Y\leq\mathbf{d}x$ then $Y\underline{\mathbf{d}}\geq x\underline{\mathbf{d}}$ as
\begin{align*}
Y\underline{\mathbf{d}}w&=\sup_{y\in Y,z\in X}(z\mathbf{d}w-z\mathbf{d}y)_+\\
&=\sup_{z\in X}(z\mathbf{d}w-\inf_{y\in Y}z\mathbf{d}y)_+\\
&=\sup_{z\in X}(z\mathbf{d}w-z\mathbf{d}Y)_+\\
&\geq\sup_{z\in X}(z\mathbf{d}w-z\mathbf{d}x)_+\\
&=x\underline{\mathbf{d}}w.
\end{align*}
Also $\underline{\mathbf{d}}\leq\mathbf{d}$, as $\mathbf{d}$ is a distance, so $Y\leq^\mathbf{d}x$ implies $Y\leq^{\underline{\mathbf{d}}}x$.

\item[\eqref{max1dir}]  If $\mathbf{d}Y=\mathbf{d}x$ then, as $Y$ is $\mathbf{d}$-final, $0=y\mathbf{d}Y=y\mathbf{d}x$, for all $y\in Y$, i.e. $Y\leq^\mathbf{d}x$ so $x=\mathbf{d}$-$\max Y$.  Conversely, as $\mathbf{d}$ is a distance, $Y\leq^\mathbf{d}x$ implies $\mathbf{d}x\leq\mathbf{d}Y+Y\mathbf{d}x=\mathbf{d}Y$.

\item[\eqref{Emax}]  If $x=\underline{\mathbf{d}}$-$\sup Y$ and $y=\mathbf{d}$-$\max Y$ then $x\underline{\mathbf{d}}y\leq Y\underline{\mathbf{d}}y\leq Y\mathbf{d}y=0$, as $\mathbf{d}$ is a distance.  So $\mathbf{d}Y=\mathbf{d}y\leq\mathbf{d}x+x\underline{\mathbf{d}}y=\mathbf{d}x$.  As $Y$ is $\mathbf{d}$-final and $Y\leq^{\underline{\mathbf{d}}}x$, $Y\mathbf{d}x\leq Y(\mathbf{d}Y)+Y\underline{\mathbf{d}}x=0$, i.e. $Y\leq^\mathbf{d}x$ too so $x=\mathbf{d}$-$\max Y$.\qedhere
\end{itemize}
\end{proof}

For any $\sqsubset\ \subseteq X\times X$, we see that
\begin{align*}
x=\text{$\sqsubset$-$\sup Y\ $}\qquad&\Leftrightarrow\qquad Y\subseteq(\sqsubset x)\quad\text{and}\quad\bigcap_{y\in Y}(y\sqsubset)\subseteq(x\sqsubset).\\
x=\text{$\sqsubset$-$\max Y$}\qquad&\Leftrightarrow\qquad Y\subseteq(\sqsubset x)\quad\text{and}\quad\bigcup_{y\in Y}(\sqsubset y)\supseteq(\sqsubset x).
\end{align*}
Thus if $\preceq$ is a partial order then $\preceq$-suprema and $\preceq$-maxima are suprema and maxima in the usual sense with respect to $\preceq$.  Indeed, if $\sqsubset$ is antisymmetric and $x\sqsubset x=\ \sqsubset$-$\max Y$ then, for some $y\in Y$, we have $x\sqsubset y\sqsubset x$ and hence $x=y$.  Maxima are more interesting for non-reflexive relations, like the way-below relation $\ll$ from domain theory or even just the strict ordering $<$ on $\mathbb{R}$.  Then maxima can be intuitively more like suprema, e.g. for any $Y\subseteq\mathbb{R}$,
\[x=\text{$<$-$\max Y$}\qquad\Leftrightarrow\qquad x=\text{$\leq$-$\sup Y$ and }x\notin Y.\]

We can also relate $\mathbf{d}$-suprema and $\mathbf{d}$-maxima to $\leq^\mathbf{d}$-suprema and $<^\mathbf{d}$-maxima, at least under certain interpolations assumptions.  One of these involves $\mathcal{P}\mathbf{d}\in[0,\infty]^{\mathcal{P}(X)\times X}$ (not to be confused with $\mathbf{d}\mathcal{P}$) defined by
\[Y(\mathcal{P}\mathbf{d})x=Y\mathbf{d}x=\sup_{y\in Y}y\mathbf{d}x.\]
So $Y\leq^{\mathcal{P}\mathbf{d}}x$ means $Y\mathbf{d}x=0$, i.e. $Y\leq^\mathbf{d}x$.

\begin{prp}\label{supmaxrelations}
For any $Y\subseteq X$,
\begin{alignat}{4}
\label{dsup=>sup}x&=\text{$\mathbf{d}$-$\sup Y$}\quad&&\Rightarrow\quad x=\text{$\leq^\mathbf{d}$-$\sup Y$}.&&\\
\label{dsup<=sup}x&=\text{$\mathbf{d}$-$\sup Y$}\quad&&\Leftarrow\quad x=\text{$\leq^\mathbf{d}$-$\sup Y$}&\qquad&\text{if }&\mathbin{\leq^{\mathcal{P}\mathbf{d}}}\circ\underline{\mathbf{d}}&\,\leq\,\mathcal{P}\mathbf{d}.\\
\label{dmax<=max}x&=\text{$\mathbf{d}$-$\max Y$}\quad&&\Leftarrow\quad x=\text{$<^\mathbf{d}$-$\max Y$}&\qquad&\text{if }&\mathbf{\overline{d}}\circ\mathbin{<^\mathbf{d}}&\,\leq\,\mathbf{d}.\\
\intertext{If $\mathbf{d}$ is a distance and $Y$ is $<^\mathbf{d}$-final then}
\label{dmax=>max}x&=\text{$\mathbf{d}$-$\max Y$}\quad&&\Rightarrow\quad x=\text{$<^\mathbf{d}$-$\max Y$}&\qquad&\text{if }&\mathbin{<^{\overline{\mathbf{d}}}}\circ\mathbin{\leq^\mathbf{d}}&\,\supseteq\,\mathbin{<^\mathbf{d}}.
\end{alignat}
\end{prp}

\begin{proof}\
\begin{itemize}
\item[\eqref{dsup=>sup}]  Multiplying $x\mathbf{d}\leq Y\mathbf{d}$ by $\infty$ yields $(x\leq^\mathbf{d})\supseteq(Y\leq^\mathbf{d})$.

\item[\eqref{dsup<=sup}]  Assume $x=\mathbin{\leq^\mathbf{d}}$-$\sup Y\neq\mathbf{d}$-$\sup Y$ so $Y\mathbf{d}z<x\mathbf{d}z$, for some $z\in X$.  As $(\mathbin{\leq^{\mathcal{P}\mathbf{d}}}\circ\underline{\mathbf{d}})\leq\mathcal{P}\mathbf{d}$, we have $w\in X$ such that $w\underline{\mathbf{d}}z<x\mathbf{d}z$ and $Y\leq^\mathbf{d}w$ and hence $x\leq^\mathbf{d}w$.  Then $x\mathbf{d}z\leq x\mathbf{d}w+w\mathbf{\underline{d}}z<x\mathbf{d}z$, a contradiction.

\item[\eqref{dmax<=max}]  Assume $x=\mathbin{<^\mathbf{d}}$-$\max Y\neq\mathbf{d}$-$\max Y$ so $z\mathbf{d}x<z\mathbf{d}Y$, for some $z\in X$.  As $(\mathbf{\overline{d}}\circ\mathbin{<^\mathbf{d}})\leq\mathbf{d}$, we have $w<^\mathbf{d}x$ with $z\mathbf{\overline{d}}w<z\mathbf{d}Y$.  This means that $w\mathbf{d}Y\geq z\mathbf{d}Y-z\mathbf{\overline{d}}w>0$ so, for all $y\in Y$, $w\nleq^\mathbf{d}y$ and hence $w\not<^\mathbf{d}y$, contradicting $x=\mathbin{<^\mathbf{d}}$-$\max Y$.

\item[\eqref{dmax=>max}]
Assume $x=\mathbf{d}$-$\max Y$.  As $Y$ is $<^\mathbf{d}$-final, for any $y\in Y$, we have $z\in Y$ with $y<^\mathbf{d}z\leq^\mathbf{d}x$ and hence $y<^\mathbf{d}x$, by \eqref{<<=}, i.e. $Y<^\mathbf{d}x$.  Now take $z\in X$ with $z<^\mathbf{d}x$.  We need to show that $z<^\mathbf{d}y$, for some $y\in Y$.  As $\mathbin{<^{\overline{\mathbf{d}}}}\circ\mathbin{\leq^\mathbf{d}}\,\supseteq\,\mathbin{<^\mathbf{d}}$, we can take $w\in X$ with $z<^{\overline{\mathbf{d}}}w\leq^\mathbf{d}x$, so $(w<^\mathbf{\overline{d}}_\epsilon)\subseteq(z\leq^{\overline{\mathbf{d}}})$, for some $\epsilon>0$.  As $w\leq^\mathbf{d}x=\mathbf{d}$-$\max Y$, we have $y\in Y$ such that $w\mathbf{\overline{d}}y\leq w\mathbf{d}y<\epsilon$ and hence $z\leq^{\overline{\mathbf{d}}}y$.  As $Y$ is $<^\mathbf{d}$-final, we have $y'\in Y$ with $y<^\mathbf{d}y'$ so $z<^\mathbf{d}y'$, by \autoref{<d}. \qedhere
\end{itemize}
\end{proof}

\section{\bf Completeness}\label{Completeness}

Next we consider generalizations of metric and directed completeness.

\begin{dfn}
For any topology $\mathcal{T}$ on $X$ and relation $\mathcal{R}\subseteq X\times\mathcal{P}(X)$,
\begin{align*}
X\text{ is \emph{$\mathbf{d}$-$\mathcal{T}\!$-complete}}\qquad&\Leftrightarrow\qquad\forall\text{$\mathbf{d}$-Cauchy }(x_\lambda)\subseteq X\ \exists x\in X(x_\lambda\xrightarrow{\mathcal{T}}x).\\
X\emph{ is $\mathbf{d}$-$\mathcal{R}$-complete}\qquad&\Leftrightarrow\qquad\forall\,\text{$\mathbf{d}$-directed }Y\subseteq X\ \exists x\mathcal{R}Y.
\end{align*}
\end{dfn}

When $\mathbf{d}$ is clear, we simply refer to $\mathcal{T}$-completeness and $\mathcal{R}$-completeness.  The cases of primary interest are $\mathcal{T}=\mathbf{d}^\circ_\circ$, $\mathbf{d}^\bullet_\circ$ and $\mathcal{R}=\mathbf{d}$-$\sup$, $\mathbf{d}$-$\max$.

When $\mathbf{d}$ is a distance and $\mathcal{T}=\mathbf{d}^\bullet_\circ$ or $\mathbf{d}^\circ_\circ$, we can replace $\mathbf{d}$-Cauchy with $\mathbf{d}$-pre-Cauchy, by \autoref{Clim} \eqref{preCauchysub} and \autoref{dlimits}.  In the hemimetric case, these are usually called Smyth and Yoneda completeness \textendash\, see \cite[Definitions 7.2.1 and 7.4.1]{Goubault2013} \textendash\, as \autoref{convchar} and \eqref{arrowc} then show that $\mathbf{d}_\bullet$-limits and $\mathbf{d}_\circ$-limits of $\mathbf{d}$-Cauchy $(x_\lambda)$ coincide.
\vspace{-10pt}
\begin{figure}[H]
\caption{Hemimetric Case}
\vspace{-10pt}
\begin{align*}
\text{Smyth complete}\quad&\Leftrightarrow\quad\mathbf{d}^\bullet_\bullet\text{-complete}\quad\Leftrightarrow\quad\mathbf{d}^\bullet_\circ\text{-complete}\\
\Rightarrow\quad\text{Yoneda complete}\quad&\Leftrightarrow\quad\mathbf{d}^\circ_\bullet\text{-complete}\quad\Leftrightarrow\quad\mathbf{d}^\circ_\circ\text{-complete}.
\end{align*}
\end{figure}
\vspace{-10pt}
\noindent
If $\mathbf{d}$ is a metric then these are all equivalent to the usual notion of metric completeness \textendash\, see \cite[Lemma 7.4.3]{Goubault2013}.

On the other hand, for any poset $(X,\sqsubseteq)$
\[\text{directed complete}\quad\Leftrightarrow\quad\text{$\sqsubseteq$-$\sup$-complete}\quad\Leftrightarrow\quad\sqsubseteq^\circ_\circ\!\text{-complete}\]
(where $\sqsubseteq^\circ_\circ$ is topology generated by $\sqsubseteq$-holes $(x\not\sqsubseteq)$ and $(\not\sqsubseteq x)$).
If $\sqsubseteq$ is the lower preorder of some transitive $\sqsubset$ on $X$ then, moreover,
\[\text{directed complete}\quad\Leftarrow\quad\text{$\sqsubset$-$\max$-complete}\quad\Leftrightarrow\quad\sqsubset^\bullet_\circ\!\text{-complete}.\]
(where $\sqsubset^\bullet_\circ$ is topology generated by upper $\sqsubset$-balls $(x\sqsubset)$ and lower $\sqsubset$-holes $(x\not\sqsubset)$).  However if $\mathbf{d}$ is a metric, every $\mathbf{d}$-directed subset contains at most $1$ element, making $X$ trivially $\mathbf{d}$-$\sup$-complete and $\mathbf{d}$-$\max$-complete.  So unlike the topological notions of completeness, the relational notions do not generalize metric completeness.  Indeed, the topological notions are stronger (even for non-distance $\mathbf{d}$), as we now show.

\begin{prp}\label{Directed=>Cauchy}
\begin{align}
\label{supcomp}X\text{ is $\mathbf{d}^\circ_\circ$-complete}\quad&\Rightarrow\quad X\text{ is $\mathbf{d}$-$\sup$-complete}.\\
\label{maxcomp}X\text{ is $\mathbf{d}^\bullet_\circ$-complete}\quad&\Rightarrow\quad X\text{ is $\mathbf{d}$-$\max$-complete}.
\end{align}
\end{prp}

\begin{proof}
Take $\mathbf{d}$-directed $Y\subseteq X$, so we have $\mathbf{d}$-Cauchy $(x_\lambda)\equiv^\mathbf{d}Y$, by \eqref{Yddir}.  Note we can take $(x_\lambda)\subseteq Y$ by taking $Y$ as the ambient space $X$ in \eqref{Yddir}.
\begin{itemize}
\item[\eqref{supcomp}]  If $X$ is $\mathbf{d}^\circ_\circ$-complete, we have $x\in X$ with $x_\lambda\carrowc x$.  As $x_\lambda\arrowc x$, \eqref{x<Yd} yields $Y\leq^\mathbf{d}x$.  As $x_\lambda\carrow x$, \eqref{cadef} and $(x_\lambda)\subseteq Y$ yield
\[x\mathbf{d}\leq\liminf_\lambda x_\lambda\mathbf{d}\leq\sup_{y\in Y}y\mathbf{d}=Y\mathbf{d}\]

\item[\eqref{maxcomp}]  If $X$ is $\mathbf{d}^\bullet_\circ$-complete, we have $x\in X$ with $x_\lambda\barrowc x$.  As $x_\lambda\arrowc x$, \eqref{x<Yd} yields $Y\leq^\mathbf{d}x$.  As $x_\lambda\barrow x$, \eqref{badef} and $(x_\lambda)\subseteq Y$ yield
\[\mathbf{d}x\geq\limsup_\lambda\mathbf{d}x_\lambda\geq\inf_{y\in Y}\mathbf{d}y=\mathbf{d}Y.\qedhere\]
\end{itemize}
\end{proof}

Conversely, we can derive the topological from the relational notions under various interpolation conditions (whose naturality/applicability will be indicated by some closely related conditions as well as examples like $C_0(X)_+$).  This was done for $\mathbf{d}^\circ_\circ$ and $\mathbf{d}$-$\sup$ in \cite{Bice2017} and here we aim to do the same for $\mathbf{d}^\bullet_\circ$ and $\mathbf{d}$-$\max$.  First we use these conditions to turn $\mathbf{d}$-pre-Cauchy nets into equivalent subsets and sequences, collecting their corollaries for completeness at the end.

Unlike much of the rest of the paper, these results have no real analogs in either metric or order theory.  Indeed, if $\mathbf{d}$ is a transitive relation $\sqsubset$ then $\sqsubset^\bullet_\circ$-completeness and $\sqsubset$-$\max$-completeness are automatically equivalent.  In this case, any $\sqsubset$-pre-Cauchy net can be turned into an equivalent $\sqsubset$-directed subset by using \autoref{Clim} \eqref{preCauchysub} to obtain a $\sqsubset$-increasing subnet (which becomes a $\sqsubset$-directed subset when we forget the indexing set).  On the other hand, as mentioned above, $\mathbf{d}$-$\max$-completeness holds trivially for any metric $\mathbf{d}$ and will thus be no help at all in verifying $\mathbf{d}^\bullet_\circ$-completeness, i.e. metric completeness.  Consequently, the results below will become either trivial or inapplicable in these classical cases.


Our first result is a converse of \eqref{x<YdpC} based on \cite[Theorem 1]{Bice2017}.  It relies on the interpolation condition $\underline{\mathbf{d}}\circ\mathbin{\leq^{\mathbf{d}\mathcal{P}}}\,\precapprox\,\mathbf{d}\mathcal{P}$ which, in the hemimetric case, weakens the middle condition considered in \autoref{minimaProp}.  This condition applies to spaces of formal balls, as we discuss in our future work, and the space $C_0(X)_+$, where again $f\mathbf{d}g=\sup_{x\in X}(f(x)-g(x))_+$.  Indeed, \eqref{minima} applies to $C_0(X)_+$, by the comments after \autoref{<dunder}, so $\underline{\mathbf{d}}\circ\mathbin{\leq^{\mathbf{d}\mathcal{P}}}\precapprox\mathbf{d}\mathcal{P}$ also applies, by \autoref{minimaProp}.  However, note that $\underline{\mathbf{d}}\circ\mathbin{\leq^{\mathbf{d}\mathcal{P}}}\,\precapprox\,\mathbf{d}\mathcal{P}$ does not apply to any metric space with at least two points.

\begin{thm}\label{Cauchytodirected}
If $\mathbf{d}$ is a distance and $\underline{\mathbf{d}}\circ\mathbin{\leq^{\mathbf{d}\mathcal{P}}}\,\precapprox\,\mathbf{d}\mathcal{P}$ then
\[(x_\lambda)\text{ is $\mathbf{d}$-pre-Cauchy}\qquad\Leftrightarrow\qquad\exists\text{ $<^\mathbf{d}$-directed }Y\equiv^\mathbf{d}(x_\lambda).\]
\end{thm}

\begin{proof}
As $\lim\limits_{r\rightarrow0}\frac{\underline{\mathbf{d}}\circ\mathbin{\leq^{\mathbf{d}\mathcal{P}}}}{\mathbf{d}\mathcal{P}}(r)=0$, we can define $r_n\downarrow0$, i.e. a strictly decreasing sequence $(r_n)$ with $r_n\rightarrow0$, such that
\[\tfrac{\underline{\mathbf{d}}\circ\mathbin{\leq^{\mathbf{d}\mathcal{P}}}}{\mathbf{d}\mathcal{P}}(2r_{n+1})<r_n.\]

Take $\mathbf{d}$-pre-Cauchy $(x_\lambda)\subseteq X$.  If necessary, we can replace $(x_\lambda)$ with a $\mathbf{d}$-Cauchy subnet, by \autoref{Clim} \eqref{preCauchysub}, and the conclusion of the theorem will be preserved, by \autoref{subnet<Y} (noting that, as $\mathbf{d}$ is a distance, any $\mathbf{d}$-pre-Cauchy net is both $\overline{\mathbf{d}}$-pre-Cauchy and $\underline{\mathbf{d}}$-pre-Cauchy, by \eqref{dis}).  Define $f:\mathcal{F}(\Lambda)\setminus\{\emptyset\}\rightarrow\Lambda$ (where $\mathcal{F}(\Lambda)$ denotes the finite subsets of $\Lambda$) recursively as follows.  Let $f(\{\lambda\})=\lambda$ and, given $F\in\mathcal{F}(\Lambda)$ with $|F|>1$, take $f(F)\succ f(E)$, for all $E\subsetneqq F$, such that
\[\sup_{f(F)\prec\lambda}x_{f(F)}\mathbf{d}x_\lambda<r_{|F|}.\]

Now $\begin{aligned}[t]
x_{f(F)}(\underline{\mathbf{d}}\circ\mathbin{\leq^{\mathbf{d}\mathcal{P}}})(x_{f(F)})^\bullet_{2r_{|F|}}&\leq\tfrac{\underline{\mathbf{d}}\circ\mathbin{\leq^{\mathbf{d}\mathcal{P}}}}{\mathbf{d}\mathcal{P}}(x_{f(F)}\mathbf{d}\mathcal{P}(x_{f(F)})^\bullet_{2r_{|F|}})\\
&\leq\tfrac{\underline{\mathbf{d}}\circ\mathbin{\leq^{\mathbf{d}\mathcal{P}}}}{\mathbf{d}\mathcal{P}}(2r_{|F|})<r_{|F|-1}.
\end{aligned}$\\
Thus we have $y_F\leq^{\mathbf{d}\mathcal{P}}(x_{f(F)})^\bullet_{2r_{|F|}}$ satisfying $x_{f(F)}\underline{\mathbf{d}}y_F<r_{|F|-1}$.  We claim that the net $(y_F)$ obtained in this way is $<^\mathbf{d}$-increasing.  Indeed, if $F\subsetneqq G$ then we can take positive $\epsilon<r_{|G|-1}-x_{f(G)}\underline{\mathbf{d}}y_G$.  If $y\in X$ satisfies $y_G\underline{\mathbf{d}}y<\epsilon$ then
\[x_{f(F)}\mathbf{d}y\leq x_{f(F)}\mathbf{d}x_{f(G)}+x_{f(G)}\underline{\mathbf{d}}y_G+y_G\underline{\mathbf{d}}y<r_{|F|}+r_{|G|-1}-\epsilon+\epsilon\leq2r_{|F|}.\]
So $(y_G<^{\underline{\mathbf{d}}}_\epsilon)\subseteq(x_{f(F)})^\bullet_{2r_{|F|}}\subseteq(y_F\leq^\mathbf{d})$, i.e. $y_F<^\mathbf{d}y_G$, proving the claim.  Thus $Y=\{y_F:F\in\mathcal{F}(\Lambda)\}$ is $<^\mathbf{d}$-directed.  Also $F\subsetneqq G$ implies
\[x_{f(F)}\mathbf{d}y_G\leq x_{f(F)}\mathbf{d}x_{f(G)}+x_{f(G)}\underline{\mathbf{d}}y_G<2r_{|F|}\rightarrow0\]
so $(x_\lambda)\leq^\mathbf{d}Y$.  And for $\lambda\succ f(F)$, $x_\lambda\in(x_{f(F)})^\bullet_{r_{|F|}}\subseteq(x_{f(F)})^\bullet_{2r_{|F|}}$ so $y_F\leq^\mathbf{d}x_\lambda$ and hence $Y\leq^\mathbf{d}(x_\lambda)$.
\end{proof}

Next we consider a different condition on balls leading to an interpolation condition involving the function $\mathcal{F}\mathbf{d}$ defined on $\mathcal{F}(X)\times X$ by
\[F(\mathcal{F}\mathbf{d})y=\sup_{x\in F}x\mathbf{d}y.\]
So $\mathcal{F}\mathbf{d}$ is just the restriction to finite subsets of $\mathcal{P}\mathbf{d}$ from \eqref{dsup<=sup}.

\begin{prp}
Every open lower $\mathbf{d}$-ball is $\leq^{\mathbf{d}}$-directed if and only if
\[\mathbin{\leq^{\mathcal{F}\mathbf{d}}}\circ\mathbf{d}\ \leq\ \mathcal{F}\mathbf{d}.\]
\end{prp}

\begin{proof}
Assume every ball $x_\bullet^r$ is $\leq^\mathbf{d}$-directed.  Then, for all finite $F\subseteq X$, $x\in X$ and $r>F\mathbf{d}x$, i.e. $F\subseteq x_\bullet^r$, we have $y\in x_\bullet^r$ with $F\leq^\mathbf{d}y$ so $\mathbin{\leq^{\mathcal{F}\mathbf{d}}}\circ\mathbf{d}\leq\mathcal{F}\mathbf{d}$.  Conversely, if $\mathbin{\leq^{\mathcal{F}\mathbf{d}}}\circ\mathbf{d}\leq\mathcal{F}\mathbf{d}$ and we have finite $F\subseteq x_\bullet^r$ then $F\mathbf{d}x<r$ so we have $y$ with $F\leq^\mathbf{d}y$ and $y\mathbf{d}x<r$, i.e. $y\in x_\bullet^r$.
\end{proof}

Again $C_0(X)_+$ with $f\mathbf{d}g=\sup_{x\in X}(f(x)-g(x))_+$ satisfies this condition, while any metric space with at least two points does not.  We weaken the condition slightly and add a completeness assumption in the following, based on \cite[Theorem 2]{Bice2017}.  Note that the result is immediate when $\mathbf{d}$ is a transitive relation $<$, as $<$-$\max$-completeness then implies that we can set each $y_n$ to be the $<$-maximum of any $<$-increasing subnet of $(x_\lambda)$.

\begin{thm}\label{CDS}
If $\mathbf{d}$ is a distance, $\mathbin{\leq^{\mathcal{F}\mathbf{d}}}\circ\mathbf{\overline{d}}\,\leq\,\mathcal{F}\mathbf{d}$ and, moreover, $X$ is $\leq^\mathbf{d}$-$(\mathbf{d}$-$\max)$-complete then
\[(x_\lambda)\text{ is $\mathbf{d}$-pre-Cauchy}\qquad\Rightarrow\qquad\exists\text{ $\underline{\mathbf{d}}^\vee\!$-Cauchy $(y_n)$ with }\mathbf{d}(x_\lambda)=\mathbf{d}(y_n).\]
\end{thm}

\begin{proof}
The basic idea of the proof will be to replace a given $\mathbf{d}$-pre-Cauchy net by one indexed by $\mathcal{F}(\Lambda)$ and then further replace this by a $\leq^\mathbf{d}$-increasing net.  The resulting limit will still be off the mark by a small amount, so we actually have to consider countably many tails of $\mathcal{F}(\Lambda)$ and replace each of the corresponding subnets by $\leq^\mathbf{d}$-increasing nets.

First note $\mathbin{\leq^{\mathcal{F}\mathbf{d}}}\circ\mathbf{\overline{d}}\leq\mathcal{F}\mathbf{d}$ is equivalent to saying that
\[f=\tfrac{\mathbin{\leq^{\mathcal{F}\mathbf{d}}}\circ\mathbf{\overline{d}}}{\mathcal{F}\mathbf{d}}\in[0,\infty]^{[0,\infty]}\]
is below the identity function on $[0,\infty]$.  This, in turn, is equivalent to saying that the $f$-image of $[0,r)$ is contained in $[0,r)$, for all $r\in(0,\infty)$.  In fact, it suffices that there are arbitrarily small such $r$.
So we assume we have $r_n\downarrow0$ with $f[0,r_n)\subseteq[0,r_n)$, for all $n\in\mathbb{N}$.  Then, for each $n$, we have positive $r^m_n\uparrow r_n$ (i.e. $\lim_mr^m_n=r_n$) with $f(r^m_n)<r^{m+1}_n$, for all $m\in\mathbb{N}$.  Taking $f(r_n^0)=0$ below, set
\[\epsilon_n^m=\tfrac{1}{2}(r_n^m-f(r_n^{m-1})).\]

Again take $\mathbf{d}$-pre-Cauchy $(x_\lambda)\subseteq X$.  Again, if necessary, we can replace $(x_\lambda)$ with a $\mathbf{d}$-Cauchy net, by \autoref{Clim} \eqref{preCauchysub}, and the conclusion of the theorem will be preserved, by \autoref{Clim} \eqref{dover} (noting that, as $\mathbf{d}$ is a distance, any $\mathbf{d}$-pre-Cauchy net $(x_\lambda)$ is $\underline{\mathbf{d}}$-pre-Cauchy, by \eqref{dis}, so $\mathbf{d}x_\lambda$ converges hence any subnet also converges to the same limit).  Define $f:\mathcal{F}(\Lambda)\rightarrow\Lambda$ recursively so $f(\{\lambda\})=\lambda$, for all $\lambda\in\Lambda$, $f(E)\prec f(F)$, for all $F\in\mathcal{F}(\Lambda)$ with $|F|>1$ and all $E\subsetneqq F$, and
\[\sup_{f(F)\prec\lambda}x_{f(F)}\mathbf{d}x_\lambda<\min_{1\leq n<|F|}\epsilon_n^{|F|-n}.\]

For any $n\in\mathbb{N}$, let $\Lambda_n=\{F\in\mathcal{F}(\Lambda):|F|>n\}$ and define $(y^n_F)_{F\in\Lambda_n}$ recursively as follows.  For $|F|=n+1$, let $y^n_F=x_{f(F)}$ so if $F\subsetneqq G$ then
\[y^n_F\mathbf{d}x_{f(G)}<\epsilon_n^1<r_n^1.\]
For $|G|=n+2$, let $Y=\{y^n_F:F\subsetneqq G\text{ and }|F|=n+1\}$.  Now
\[Y(\mathbin{\leq^{\mathcal{F}\mathbf{d}}}\circ\mathbf{\overline{d}})x_{f(G)}\leq\tfrac{\mathbin{\leq^{\mathcal{F}\mathbf{d}}}\circ\mathbf{\overline{d}}}{\mathcal{F}\mathbf{d}}(Y\mathbf{d}x_{f(G)})\leq f(r_n^1)\]
so we can take $y^n_G$ with
\[Y\leq^\mathbf{d}y^n_G\quad\text{ and }\quad y^n_G\overline{\mathbf{d}}x_{f(G)}<f(r_n^1)+\epsilon^2_n.\]
As $x_{f(G)}\mathbf{d}x_{f(H)}<\epsilon^2_n$, whenever $G\subsetneqq H$ and $|G|=n+2$,
\[y^n_G\mathbf{d}x_{f(H)}\leq y^n_G\overline{\mathbf{d}}x_{f(G)}+x_{f(G)}\mathbf{d}x_{f(H)}<f(r_n^1)+2\epsilon^2_n=r_n^2.\]
Note that we also have
\[x_{f(G)}\mathbf{d}x_{f(H)}<\epsilon_n^2<r_n^2.\]
Thus if $|H|=n+3$ and $Z=\bigcup\{\{y^n_G,x_{f(G)}\}:G\subsetneqq H\text{ and }|G|=n+2\}$,
\[Z(\mathbin{\leq^{\mathcal{F}\mathbf{d}}}\circ\mathbf{\overline{d}})x_{f(H)}\leq\tfrac{\mathbin{\leq^{\mathcal{F}\mathbf{d}}}\circ\mathbf{\overline{d}}}{\mathcal{F}\mathbf{d}}(Z\mathbf{d}x_{f(H)})\leq f(r_n^2).\]
Thus we can take $y^n_H$ with $Z\leq^\mathbf{d}y^n_H$ and
\[y^n_H\overline{\mathbf{d}}x_{f(H)}<f(r_n^2)+\epsilon^3_n.\]
Continuing in this way we obtain $\leq^\mathbf{d}$-increasing $(y^n_F)$ with $y^n_F\mathbf{d}x_{f(G)}<r_n$ and $x_{f(F)}\leq^\mathbf{d}y^n_G$, for all $F\in\Lambda_n$ and $F\subsetneqq G$.

As $X$ is $\leq^\mathbf{d}$-$\mathbf{d}$-$\max$-complete, we can take $y^n=\mathbf{d}$-$\max_Fy^n_F$.  Thus $x_{f(F)}\leq^\mathbf{d}y^n$, for all $F\in\Lambda_n$.  As $(x_\lambda)$ is $\mathbf{d}$-Cauchy, \autoref{Clim} \eqref{dunder} implies that $x_\lambda\mathbf{d}$ converges.  As $(x_{f(F)})_{F\in\Lambda_n}$ is a subnet of $(x_\lambda)$, for each $n\in\mathbb{N}$ we have
\[\lim_\lambda x_\lambda\mathbf{d}y^n=\lim_{F\in\Lambda_n}x_F\mathbf{d}y^n=0.\]
Thus
\[\mathbf{d}(y^n)\leq\liminf_n\liminf_\lambda(\mathbf{d}x_\lambda+x_\lambda\mathbf{d}y^n)=\mathbf{d}(x_\lambda).\]
 $\mathbf{d}(y^n)\leq\mathbf{d}(x_\lambda)$.  Also, for any $m,n\in\mathbb{N}$ and all sufficiently large $H$, $G$ and $F$, specifically $H\supsetneqq G\supsetneqq F\in\Lambda_{m\vee n}$,
\[y^m_F\mathbf{d}y^n\leq y^m_F\mathbf{d}y^n_H\leq y^m_F\mathbf{d}x_{f(G)}<r_m.\]
By \eqref{max=>sup}, $y^m=\underline{\mathbf{d}}$-$\sup_F y^m_F$ so $y^m\underline{\mathbf{d}}y^n\leq\lim_Fy^m_F\underline{\mathbf{d}}y^n\leq r_m$ and hence $y^m\underline{\mathbf{d}}^\vee y^n\leq r_{m\wedge n}$.  As $r_n\rightarrow0$, this shows that $(y^n)$ is $\underline{\mathbf{d}}^\vee$-Cauchy.  Also, as $(x_\lambda)$ is $\mathbf{d}$-Cauchy, $\mathbf{d}x_\lambda$ converges, by \autoref{Clim} \eqref{dover}, so
\begin{align*}
\mathbf{d}(x_\lambda)&=\lim_\lambda\mathbf{d}x_\lambda\\
&=\lim_G\mathbf{d}x_{f(G)}\\
&\leq\liminf_n\liminf_F\liminf_G(\mathbf{d}y^n_F+y^n_F\mathbf{d}x_{f(G)})\\
&\leq\liminf_n(\mathbf{d}y^n+r_n)\qquad\text{by \eqref{maxdef}}\\
&=\mathbf{d}(y^n).\qedhere
\end{align*}
\end{proof}

Another natural interpolation condition involves the symmetrization:
\[\mathbf{d}^\vee\circ\mathbin{\leq^\mathbf{d}}\leq\mathbf{d}.\]
Yet again this is satisfied by $C_0(X)_+$ where $f\mathbf{d}g=\sup_{x\in X}(f(x)-g(x))_+$.  More interestingly, even in non-commutative C*-algebras we have the weaker uniform interpolation condition $\mathbf{d}^\vee\circ\mathbin{\leq^\mathbf{d}}\precapprox\mathbf{d}$ on the positive unit ball, where $a\mathbf{d}b=\|(a-b)_+\|$ (see \cite[Theorem 2.6]{BiceVignati2018}).  For the result itself, based on \cite[Theorem 3]{Bice2017}, it again suffices to consider an even weaker condition, this time involving a modification of $\mathbf{e}\ \circ\leq^{\mathbf{d}}$ defined by
\[\mathbf{e}\circ\Phi^\mathbf{d}=\sup_{n\in\mathbb{N}}(\mathbf{e}\circ n\mathbf{d})=\sup\limits_{\epsilon>0}(\mathbf{e}\circ\mathbin{<^\mathbf{d}_\epsilon}).\]
In particular, note $(\mathbf{e}\circ\Phi^\mathbf{d})\leq(\mathbf{e}\circ\mathbin{\leq^\mathbf{d}})$.  Also note that when $\mathbf{d}$ is a metric, $Y$ below must be a singleton set $\{x\}$ with $x_\lambda\rightarrow x$.  In this case, the result is really just saying that limits coincide for uniformly equivalent metrics.

\begin{thm}\label{ded}
If $\mathbf{d}$ and $\mathbf{e}$ are distances, $X$ is $\mathbf{e}_\circ$-complete, $\mathbf{e}\circ\Phi^{\overline{\mathbf{d}}}\precapprox\mathbf{d}$ and $\underline{\mathbf{d}},\overline{\mathbf{d}}^\mathrm{op}\precapprox\mathbf{e}$ then
\[(x_\lambda)\text{ is $\mathbf{d}$-pre-Cauchy}\qquad\Leftrightarrow\qquad\exists\text{ (necessarily $\mathbf{d}$-directed) }Y\equiv^\mathbf{d}(x_\lambda).\]
\end{thm}

\begin{proof}
Given $\mathbf{d}$-pre-Cauchy $(x_\lambda)$, we may again take a subnet indexed by $\mathcal{F}(\Lambda)$ if necessary and assume we have nets $(s_\lambda),(t_\lambda)\subseteq(0,\infty)$ such that
\begin{align}
\label{slambda}\sup_{\lambda\prec\delta}x_\lambda\mathbf{d}x_\delta&<s_\lambda\rightarrow0.\\
\label{tlambda}\tfrac{\mathbf{e}\circ\Phi^{\overline{\mathbf{d}}}}{\mathbf{d}}(s_\lambda)&<t_\lambda\rightarrow0.
\end{align}
For each $\lambda$, we define $\gamma_\lambda^n$ and $x_\lambda^n$ recursively so that
\begin{align*}
x_\lambda^n\overline{\mathbf{d}}x_{\gamma_\lambda^n}+\sup_{\gamma_\lambda^n\prec\delta}x_{\gamma_\lambda^n}\mathbf{d}x_\delta<s_{\gamma_\lambda^n}&<2^{1-n}t_\lambda.\\
\tfrac{\mathbf{e}\circ\Phi^{\overline{\mathbf{d}}}}{\mathbf{d}}(s_{\gamma_\lambda^n})&<2^{1-n}t_\lambda.\\
x^n_\lambda\mathbf{e}x^{n+1}_\lambda&<2^{1-n}t_\lambda.
\end{align*}
First set $\gamma_\lambda^1=\lambda$ and $x_\lambda^1=x_\lambda$.  For $n\in\mathbb{N}$, take $\gamma_\lambda^{n+1}\succ\gamma_\lambda^n$ with $\tfrac{\mathbf{e}\circ\Phi^{\overline{\mathbf{d}}}}{\mathbf{d}}(s_{\gamma_\lambda^{n+1}}),s_{\gamma_\lambda^{n+1}}<2^{-n}t_\lambda$.  As $x_\lambda^n\mathbf{d}x_{\gamma_\lambda^{n+1}}\leq x_\lambda^n\overline{\mathbf{d}}x_{\gamma_\lambda^n}+x_{\gamma_\lambda^n}\mathbf{d}x_{\gamma_\lambda^{n+1}}<s_{\gamma_\lambda^n}$,
\[x_\lambda^n(\mathbf{e}\circ\Phi^{\overline{\mathbf{d}}})x_{\gamma_\lambda^{n+1}}\leq\tfrac{\mathbf{e}\circ\Phi^{\overline{\mathbf{d}}}}{\mathbf{d}}(x_\lambda^n\mathbf{d}x_{\gamma_\lambda^{n+1}})\leq\tfrac{\mathbf{e}\circ\Phi^{\overline{\mathbf{d}}}}{\mathbf{d}}(s_{\gamma_\lambda^n})<2^{1-n}t_\lambda,\]
so we can take $x^{n+1}_\lambda$ with $x^n_\lambda\mathbf{e}x^{n+1}_\lambda<2^{1-n}t_\lambda$ and
\[x_\lambda^{n+1}\overline{\mathbf{d}}x_{\gamma_\lambda^{n+1}}<s_{\gamma_\lambda^{n+1}}-\sup\limits_{\gamma_\lambda^{n+1}\prec\delta}x_{\gamma_\lambda^{n+1}}\mathbf{d}x_\delta.\]
Note the right side above is positive by \eqref{slambda} (with $\gamma_\lambda^{n+1}$ in place of $\lambda$).  Thus the recursion may continue.

For each $\lambda$, $x^n_\lambda\mathbf{e}x^{n+1}_\lambda<2^{1-n}t_\lambda$ so $(x_\lambda^n)_{n\in\mathbb{N}}$ is $\mathbf{e}$-Cauchy.  As $X$ is $\mathbf{e}_\circ$-complete, we have $y_\lambda\in X$ with $\lim_nx_\lambda^n\mathbf{e}y_\lambda=0$, by \eqref{arrowc}, and hence $\lim_ny_\lambda\overline{\mathbf{d}}x_\lambda^n=0$, as $\overline{\mathbf{d}}^\mathrm{op}\precapprox\mathbf{e}$.  Now
\begin{align*}
\limsup_\delta y_\lambda\mathbf{d}x_\delta&\leq\liminf_n\limsup_\delta(y_\lambda\overline{\mathbf{d}}x_\lambda^n+x_\lambda^n\overline{\mathbf{d}}x_{\gamma_\lambda^n}+x_{\gamma_\lambda^n}\mathbf{d}x_\delta)\\
&\leq\liminf_n(y_\lambda\overline{\mathbf{d}}x_\lambda^n+s_{\gamma_\lambda^n})\\
&\leq\liminf_n(y_\lambda\overline{\mathbf{d}}x_\lambda^n+2^{1-n}t_\lambda)\\
&=0.
\end{align*}
So $Y\leq^\mathbf{d}(x_\lambda)$ for $Y=\{y_\lambda:\lambda\in\Lambda\}$.  As $x_\lambda=x^1_\lambda$ and $x^n_\lambda\mathbf{e}x^{n+1}_\lambda<2^{1-n}t_\lambda$, $x_\lambda\mathbf{e}y_\lambda\leq2t_\lambda\rightarrow0$.  Thus $x_\lambda\underline{\mathbf{d}}y_\lambda\rightarrow0$, as $\underline{\mathbf{d}}\precapprox\mathbf{e}$.  Now
\begin{align*}
x_\lambda\mathbf{d}Y&=\inf_{y\in Y}x_\lambda\mathbf{d}y\\
&\leq\limsup_\delta x_\lambda\mathbf{d}y_\delta\\
&\leq\limsup_\delta(x_\lambda\mathbf{d}x_\delta+x_\delta\underline{\mathbf{d}}y_\delta)\\
&=\limsup_\delta(x_\lambda\mathbf{d}x_\delta)\qquad\text{as }x_\delta\underline{\mathbf{d}}y_\delta\rightarrow0\\
&\rightarrow0\qquad\text{as $(x_\lambda)$ is $\mathbf{d}$-pre-Cauchy}.
\end{align*}
Thus $(x_\lambda)\leq^\mathbf{d}Y$ and hence $Y$ is $\mathbf{d}$-directed, by \eqref{x<Yddir}.
\end{proof}

Replacing $\Phi^{\overline{\mathbf{d}}}$ with $\Phi^\mathbf{d}$, we get $\leq^\mathbf{d}$-directed subsets from $\mathbf{d}$-pre-Cauchy sequences (rather than $\mathbf{d}$-directed subsets from $\mathbf{d}$-pre-Cauchy nets).  In fact, as the subset $Y$ is countable, it could even be replaced with a cofinal increasing sequence.  Indeed, this is how $Y$ is constructed in the proof, which is based on the argument given in \cite[Theorem 4.5]{Bice2017}.

\begin{thm}\label{ded2}
If $\mathbf{d}$ and $\mathbf{e}$ are distances, $X$ is $\mathbf{e}_\circ$-complete, $\mathbf{e}\circ\Phi^\mathbf{d}\precapprox\mathbf{d}$ and $\underline{\mathbf{d}},\overline{\mathbf{d}}^\mathrm{op}\precapprox\mathbf{e}$ then $\mathbf{e}\circ\Phi^\mathbf{d}=\mathbf{e}\circ\mathbin{\leq^\mathbf{d}}$ and
\[(x_n)_{n\in\mathbb{N}}\text{ is $\mathbf{d}$-pre-Cauchy}\qquad\Leftrightarrow\qquad\exists\text{ (countable) $\leq^\mathbf{d}$-directed }Y\equiv^\mathbf{d}(x_n).\]
\end{thm}

\begin{proof}
First we prove $\mathbf{e}\circ\Phi^\mathbf{d}=\mathbf{e}\circ\mathbin{\leq^\mathbf{d}}$.  For any $x,y\in X$ and $\epsilon>0$, take $\epsilon_n\downarrow0$ with $\frac{\mathbf{e}\circ\Phi^\mathbf{d}}{\mathbf{d}}(\epsilon_n)<2^{-n}\epsilon$, for all $n\in\mathbb{N}$.  Now take $z_1\in X$ with $x\mathbf{e}z_1<x(\mathbf{e}\circ\Phi^\mathbf{d})y+\epsilon$ and $z_1\mathbf{d}y<\epsilon_1$.  Thus
\[z_1(\mathbf{e}\circ\Phi^\mathbf{d})y\leq\tfrac{\mathbf{e}\circ\Phi^\mathbf{d}}{\mathbf{d}}(z_1\mathbf{d}y)\leq\tfrac{\mathbf{e}\circ\Phi^\mathbf{d}}{\mathbf{d}}(\epsilon_1)<\tfrac{1}{2}\epsilon\]
and we can take $z_2\in X$ such that $z_1\mathbf{e}z_2<\frac{1}{2}\epsilon$ and $z_2\mathbf{d}y<\epsilon_2$.  Continuing in this way we obtain a sequence $(z_n)\subseteq X$ such that, for all $n\in\mathbb{N}$,
\[z_n\mathbf{e}z_{n+1}\leq2^{-n}\epsilon\qquad\text{and}\qquad z_n\mathbf{d}y<\epsilon_n\rightarrow0.\]
As $\mathbf{e}$ is a distance and $X$ is $\mathbf{e}_\circ$-complete, \eqref{arrowc} yields $z_n\mathbf{e}z\rightarrow0$, for some $z\in X$, so
\[x\mathbf{e}z\leq x\mathbf{e}z_1+z_1\mathbf{e}z\leq x(\mathbf{e}\circ\Phi^\mathbf{d})y+2\epsilon.\]
Also $z\mathbf{d}y\leq z\overline{\mathbf{d}}z_n+z_n\mathbf{d}y\leq z\overline{\mathbf{d}}z_n+\epsilon_n\rightarrow0$, as $\overline{\mathbf{d}}^\mathrm{op}\precapprox\mathbf{e}$ and $z_n\mathbf{e}z\rightarrow0$, so $z\leq^\mathbf{d}y$. As $\epsilon>0$ was arbitrary, $(\mathbf{e}\circ\mathbin{\leq^\mathbf{d}})\leq(\mathbf{e}\circ\Phi^\mathbf{d})$.  The reverse inequality is immediate.

Now take $(s^m_n),(t^m_n)\subseteq(0,\infty)$ such that, for all $m,n\in\mathbb{N}$,
\[s^m_n<2^{-m-n},\quad\tfrac{\underline{\mathbf{d}}}{\mathbf{e}}(s^m_{n+1})<t^m_n\quad\text{and}\quad\tfrac{\mathbf{e}\circ\mathbin{\leq^\mathbf{d}}}{\mathbf{d}}(t^m_n)<s^{m+1}_n\]
(define $(s^m_1)_{m\in\mathbb{N}}$ first then $(t^m_1)_{m\in\mathbb{N}}$, $(s^m_2)_{m\in\mathbb{N}}$ etc., also note that the top of $\tfrac{\underline{\mathbf{d}}}{\mathbf{e}}$ above is $\underline{\mathbf{d}}$, not $\mathbf{d}$).  Take a subsequence of the given $\mathbf{d}$-pre-Cauchy $(x_n)$ with $x_n\mathbf{d}x_{n+1}<t^1_n$, for all $n$, and define $y^m_n$ with $y^m_n\mathbf{d}y^m_{n+1}<t^m_n$, for all $m$ and $n$, recursively as follows.  First let $y^1_n=x_n$, for all $n$.  Assume $y^m_n$ is defined for all $n$ and fixed $m$.  For each $n$, we can take $y^{m+1}_n\leq^\mathbf{d}y^m_{n+1}$ with $y^m_n\mathbf{e}y^{m+1}_n<s^{m+1}_n$ as
\[y^m_n(\mathbf{e}\circ\mathbin{\leq^\mathbf{d}})y^m_{n+1}\leq\tfrac{\mathbf{e}\circ\mathbin{\leq^\mathbf{d}}}{\mathbf{d}}(y^m_n\mathbf{d}y^m_{n+1})\leq\tfrac{\mathbf{e}\circ\mathbin{\leq^\mathbf{d}}}{\mathbf{d}}(t^m_n)<s^{m+1}_n.\]
Then the recursion may continue because
\begin{align*}
y^{m+1}_n\mathbf{d}y^{m+1}_{n+1}&\leq y^{m+1}_n\mathbf{d}y^m_{n+1}+y^m_{n+1}\underline{\mathbf{d}}y^{m+1}_{n+1}=y^m_{n+1}\underline{\mathbf{d}}y^{m+1}_{n+1}\\
&\leq\tfrac{\underline{\mathbf{d}}}{\mathbf{e}}(y^m_{n+1}\mathbf{e}y^{m+1}_{n+1})\leq\tfrac{\underline{\mathbf{d}}}{\mathbf{e}}(s^{m+1}_{n+1})<t^{m+1}_n
\end{align*}

For all $m,n\in\mathbb{N}$, $y^m_n\mathbf{e}y^{m+1}_n<s_n^{m+1}<2^{-m-1-n}<2^{-m-n}$ so, as $X$ is $\mathbf{e}_\circ$-complete, we have $y_n\in X$ with $\lim_my^m_n\mathbf{e}y_n=0$.  As $\underline{\mathbf{d}},\overline{\mathbf{d}}^\mathrm{op}\precapprox\mathbf{e}$ and $y^{m+1}_n\leq^\mathbf{d}y^m_{n+1}$,
\[y_n\mathbf{d}y_{n+1}\leq\liminf_m(y_n\overline{\mathbf{d}}y^{m+1}_n+y^{m+1}_n\mathbf{d}y^m_{n+1}+y^m_{n+1}\underline{\mathbf{d}}y_{n+1})=0,\]
i.e. $y_n\leq^\mathbf{d}y_{n+1}$ so $Y=\{y_n:n\in\mathbb{N}\}$ is $\leq^\mathbf{d}$-directed.  Also, again using the fact that $\mathbf{e}$ is a distance, we have
\[x_n\mathbf{e}y_n\leq\liminf_m(x_n\mathbf{e}y^m_n+y^m_n\mathbf{e}y_n)<{\textstyle\sum\limits_{m=2}^\infty}s^m_n<{\textstyle\sum\limits_{m=2}^\infty}2^{-m-n}<2^{-n}\rightarrow0.\]
This together with $\underline{\mathbf{d}}\precapprox\mathbf{e}$ and the fact that $(x_n)$ is $\mathbf{d}$-pre-Cauchy yields
\[x_n\mathbf{d}Y=\inf_mx_n\mathbf{d}y_m\leq\liminf_m(x_n\mathbf{d}x_m+x_m\underline{\mathbf{d}}y_m)=\liminf_m x_n\mathbf{d}x_m\rightarrow0,\]
so $(x_n)\leq^\mathbf{d}Y$.  Likewise, using $\overline{\mathbf{d}}^\mathrm{op}\precapprox\mathbf{e}$ instead and the fact $\mathbf{d}$ is a distance,
\begin{align*}
\limsup_my_n\mathbf{d}x_m&\leq\liminf_l\limsup_m(y_n\mathbf{d}y_l+y_l\mathbf{d}x_m)\\
&\leq\liminf_l\limsup_m(y_n\mathbf{d}y_l+y_l\overline{\mathbf{d}}x_l+x_l\mathbf{d}x_m)=0
\end{align*}
so $Y\leq^\mathbf{d}(x_n)$.
\end{proof}

As promised, we can now show that $\mathbf{d}^\bullet_\circ$-completeness follows from $\mathbf{d}$-$\max$-completeness (or even slightly weaker notions) under various additional interpolation, completeness and separability conditions.

\begin{cor}\label{Sc} $X$ is $\mathbf{d}^\bullet_\circ$-complete if $\mathbf{d}$ and $\mathbf{e}$ are distances satisfying any of the following ($\mathbf{d}$-$\mathcal{R}$-$\mathcal{T}$-complete means $\mathbf{d}$-$\mathcal{R}$-complete and $\mathcal{T}$-complete).
\begin{align}
\label{Sc1}\underline{\mathbf{d}}\circ\mathbin{\leq^{\mathbf{d}\mathcal{P}}}\,&\precapprox\,\mathbf{d}\mathcal{P}&\text{and}&&X\text{ is $<^\mathbf{d}$-$(\mathbf{d}$-$\max)$-complete.}\hspace{14pt}&\\
\label{Sc2}\mathbin{\leq^{\mathcal{F}\mathbf{d}}}\circ\mathbf{\overline{d}}\,&\leq\,\mathcal{F}\mathbf{d}&\text{and}&&X\text{ is $\leq^\mathbf{d}$-$(\mathbf{d}$-$\max)$-$\underline{\mathbf{d}}^\vee_\circ\!$-complete.}&\\
\label{Sc3}\mathbf{e}\circ\Phi^{\overline{\mathbf{d}}}\,&\precapprox\,\mathbf{d},&\underline{\mathbf{d}},\overline{\mathbf{d}}^\mathrm{op}&\precapprox\mathbf{e}&\&\ X\text{ is $\mathbf{d}$-$(\mathbf{d}$-$\max)$-$\mathbf{e}_\circ$-complete.}&\\
\label{Sc4}\mathbf{e}\circ\Phi^\mathbf{d}\,&\precapprox\,\mathbf{d},&\underline{\mathbf{d}},\overline{\mathbf{d}}^\mathrm{op}&\precapprox\mathbf{e},&X\text{ is $\leq^\mathbf{d}$-$(\mathbf{d}$-$\max)$-$\mathbf{e}_\circ$-complete}\hspace{4pt}&\\
\nonumber&&&&\text{and $\overline{\mathbf{d}}^\bullet_\bullet$-separable.}&
\end{align}
\end{cor}

\begin{proof}
Take $\mathbf{d}$-Cauchy $(x_\lambda)$.
\begin{itemize}
\item[\eqref{Sc1}] By \autoref{Cauchytodirected}, we have $<^\mathbf{d}$-directed $Y$ such that $Y\equiv^\mathbf{d}(x_\lambda)$ and hence $\mathbf{d}Y=\mathbf{d}(x_\lambda)$, by \eqref{dis} and \autoref{Yxlam}.  By $<^\mathbf{d}$-$(\mathbf{d}$-$\max)$-completeness and \eqref{max1dir}, we have $x\in X$ with $\mathbf{d}x=\mathbf{d}Y=\mathbf{d}(x_\lambda)$ so $x_\lambda\barrowc x$, by \eqref{bac}.  Thus $X$ is $\mathbf{d}^\bullet_\circ$-complete.

\item[\eqref{Sc2}] By \autoref{CDS}, we have $\underline{\mathbf{d}}^\vee$-Cauchy $(y_n)$ with $\mathbf{d}(x_\lambda)=\mathbf{d}(y_n)$.  By $\underline{\mathbf{d}}^\vee_\circ$-completeness and \eqref{arrowc}, we have $x\in X$ with $y_n\underline{\mathbf{d}}^\vee x\rightarrow0$ and hence $y_n\barrowc x$, by \eqref{abvee}.  Thus $\mathbf{d}x=\mathbf{d}(y_n)=\mathbf{d}(x_\lambda)$ and hence $x_\lambda\barrowc x$, by \eqref{bac}.  Thus $X$ is $\mathbf{d}^\bullet_\circ$-complete.

\item[\eqref{Sc3}] By \autoref{ded}, we have $\mathbf{d}$-directed $Y\equiv^\mathbf{d}(x_\lambda)$ and hence $\mathbf{d}Y=\mathbf{d}(x_\lambda)$, by \eqref{dis} and \autoref{Yxlam}.  By $\mathbf{d}$-$\max$-completeness and \eqref{max1dir}, we have $x\in X$ with $\mathbf{d}x=\mathbf{d}Y=\mathbf{d}(x_\lambda)$ so $x_\lambda\barrowc x$, by \eqref{baclim}.  Thus $X$ is $\mathbf{d}^\bullet_\circ$-complete.

\item[\eqref{Sc4}] By \autoref{ded}, we have $\mathbf{d}$-directed $Y\equiv^\mathbf{d}(x_\lambda)$.  By \eqref{sepdir}, we have $(x'_n)_{n\in\mathbb{N}}\equiv^\mathbf{d}Y$.  By \autoref{ded2}, we have $\leq^\mathbf{d}$-directed $Y'\equiv^\mathbf{d}(x'_n)$ and hence $\mathbf{d}Y'=\mathbf{d}(x_\lambda)$, by \eqref{dis} and \autoref{Yxlam}.  By $\leq^\mathbf{d}$-$(\mathbf{d}$-$\max)$-completeness, we have $x\in X$ with $\mathbf{d}x=\mathbf{d}Y'=\mathbf{d}(x_\lambda)$, i.e. $x_\lambda\barrowc x$.  Thus $X$ is $\mathbf{d}^\bullet_\circ$-complete.\qedhere
\end{itemize}
\end{proof}

\newpage

\bibliography{maths}{}
\bibliographystyle{alphaurl}

\end{document}